\documentclass[12pt, journal]{IEEEtran}
\IEEEoverridecommandlockouts
\usepackage[margin=0.75in, paperwidth = 9in, paperheight=12in]{geometry}
\usepackage[vlined, ruled, boxed]{algorithm2e}
\usepackage{graphicx}
\usepackage{hyperref}
\hypersetup{colorlinks, 
	linkcolor = red,
	citecolor = blue
}
\usepackage[skip = 0.75em]{parskip}
\usepackage{microtype}
\newcommand{\asympl}{\ensuremath{\;\dot{\sim}\;}}
\newcommand{\nasympl}{\ensuremath{\;\dot{\nsim}\;}}
\newlength{\bibitemsep}\newlength{\bibparskip}\let\oldthebibliography\thebibliography\renewcommand\thebibliography[1]{\oldthebibliography{#1}\setlength{\parskip}{\bibitemsep}\setlength{\itemsep}{\bibparskip}}
\usepackage{mathrsfs}
\usepackage[utf8]{inputenc}
\usepackage{cite}
\usepackage{mathtools}
\usepackage{enumerate}
\usepackage{mathrsfs}
\usepackage{url}
\usepackage{dirtytalk}
\usepackage{amsmath,amsthm, amssymb,amsfonts}
\DeclareMathOperator{\sgn}{sgn}
\usepackage{hyperref}
\usepackage{authblk}
\usepackage{parskip}
\usepackage{setspace}
\usepackage[english]{babel}
\newtheorem{theorem}{Theorem}
\newtheorem{lemma}{Lemma}
\newtheorem{corollary}{Corollary}
\newtheorem{remark}{Remark}
\newtheorem{assumption}{Assumption}
\theoremstyle{definition}
\newtheorem*{defn*}{Definition}
\renewcommand\thesection{\arabic{section}}
\DeclareFontFamily{U}{calligra}{}
\DeclareFontShape{U}{calligra}{m}{n}{<->callig15}{}

\begin{document}
\onecolumn
\title{Whiplash Gradient Descent Dynamics}
\author{Subhransu~S.~Bhattacharjee$^*$~and~Ian~R.~Petersen$^*$
\thanks{*Both authors were with the School of Engineering, The Australian National University at the time of writing this paper. Subhransu is the corresponding author for this paper. Please contact him if required at \href{mailto:Subhransu.Bhattacharjee@anu.edu.au}{Subhransu.Bhattacharjee@anu.edu.au}}.}
\maketitle
\begin{abstract}
In this paper, we propose the Whiplash Inertial Gradient dynamics, a closed-loop optimization method that utilises gradient information, to find the minima of a cost function in finite-dimensional settings. We introduce the symplectic asymptotic convergence analysis for the Whiplash system for convex functions. We also introduce relaxation sequences to explain the non-classical nature of the algorithm and an exploring heuristic variant of the Whiplash algorithm to escape saddle points, deterministically. We study the algorithm's performance for various costs and provide a practical methodology for analyzing convergence rates using integral constraint bounds and a novel Lyapunov rate method. Our results demonstrate polynomial and exponential rates of convergence for quadratic cost functions.
\end{abstract}
\section{Introduction}
In this paper, we study the continuous optimization of finite-dimensional, unconstrained problems. We revisit classical optimization theories at the heart of popular deep learning algorithms. While studying these algorithms, we consider finite-dimensional global minimization problems with optima at $x^*\in\mathcal{X}$, such that in finite dimension $d>0$, we have objective costs of the form $$f^* = \min_{\mathcal{X}\in\mathbb{R}^d} f(x),$$ where $f:\mathbb{R}^d\mapsto\mathbb{R}$. These problems arise in fields such as deep learning, economics, and physics. With stochastic gradient approaches facing information bottlenecks \cite{bottle}, it is important to revisit deterministic approaches to understand ways to improve modern optimization tools from both practical and theoretical perspectives. Second-order methods which are significantly faster \cite{beck} often tend to be computationally infeasible. Thus, we look back at first-order gradient methods which are computationally cheap and deterministic in nature. 

Treating optimization algorithms as differential equations have been done in the past, where the stable points of the gradient flow system are the optimal points of the cost function \cite{odeo}. In our approach, we make use of non-linear systems theory frameworks to analyze algorithms in continuous-time as well as tuning their behaviour in discrete-time in order to introduce novel methods with wider practicality in the context of systems theory and numerical analysis. In order to systematise the treatment we primarily consider the gradient function as the black box oracle for continuous functions introduced by Nemirovsky \cite{Nem}.

\subsection{Motivation}

In the field of continuous optimization, black box models are prevalent in settings with scarce computational resources where learning about the global geometry of a function is not feasible. In most situations, a descent direction is necessary for a gradient descent method, making the computation of a function's gradient necessary. Beyond the computational complexities involved in finding the gradient of the function itself, black box methods allow for finding mathematically optimal convergence of iterative processes used for optimization. While global optimization is the true goal for such methods, the optimization of convex functions is inherently simpler because finding a single minimum within a comparatively uniform geometry is achievable using just the gradient information.

In black box methods, we consider computational models that make queries to an oracle that contain a finite-dimensional vector map of the entire linear span of gradients of a cost function \cite{Nem}. While various methods of convex optimization has evolved, including the optimal convex method, by accelerating the process, as shown by Nesterov \cite{Nest}, it still happens to be an area of active research. One of the reasons for the use of black box models in continuous optimization is to handle the variety of non-convex problems encountered in the context of deep learning and robotics. Non-convex problems require a diverse set of approaches, but assuming local convexity is the predominant way to prove convergence in most scenarios. While modern optimization uses stochastic methods to solve such problems \cite{primer}, deterministic methods are constructive when it comes to designing new methods for optimization, as the convergence analysis is far simpler. 

Treating optimization methods as dynamical systems allows for tuning the parameters in order to scale the flow leading to faster convergence \cite{odeo}, which can later be converted to discrete algorithms using discretization methods. A variational approach developed by Wilson et al. \cite{WILL} describes a generalized Lyapunov analysis framework for continuous-time systems. They showed that this technique of estimate sequences is equivalent to a family of Lyapunov functions in both continuous and discrete-time. This framework has led to the use of carefully designed time-scaled energy functions as a standard method to prove convergence rates for optimization algorithms, which can be studied as continuous-time inertial gradient flow dynamical systems. Shi et al. \cite{Shi} demonstrate that it is impractical to use discrete schemes to find discrete Lyapunov functions. Since, Su et al. \cite{su2015differential} showed the Nesterov’s method had a dynamical system equivalent, we know that treating optimization algorithms as dynamical systems make them easier to analyse their properties in high-resolution. Although it is much more feasible in continuous time, the task of finding such  energy functions is nevertheless complicated. Depending upon the structure of the terms involved (often such terms are time-varying and non-linear), and can only be derived vaguely using physical energy-based motivations. As a method to bypass this problem, Attouch et al. \cite{attouch2021fast} introduced a closed-loop method as a heuristic design based on physical analogies. They used a parametric manufactured solution for a particular cost function to consider a specific family of dynamical systems. However, this method is limited from a control-theoretic perspective as it does not provide any energy function to prove the cost function's convergence rate. 

The task of finding appropriate Lyapunov functions becomes challenging for closed-loop ODEs, which use damping coefficients because the damping terms depend on the system's dynamics and not just on the function, leading to non-classical behavior such as accelerated phases followed by attenuating oscillations. Our study aims to provide a generalized numerical energy-based framework to find the strength of convergence for unconstrained convex optimization problems using a certain class of closed-loop gradient flow control systems using nonlinear feedback systems theory.

\subsection{Contribution} 
This paper is an extended culmination of the conference papers: \cite{Whiplash} and \cite{ACC}. In \cite{Whiplash}, we introduced the Whiplash Inertial Gradient dynamics as born out of the motivation to treat optimization algorithms as dynamical systems and perform control system analyses to find a novel method that is non-classical in nature from the accelerated gradient methods. It is significantly faster than other deterministic algorithms by a magnitude of $10^5$ for the same choice of step-size for the Rosenbrock's function. In \cite{ACC}, we introduced relaxation sequences, which defined the characteristics of the Whiplash dynamics in discrete-time and the novel convergence rate methodology, the Lyapunov Rate method, which is an alternative to finding convergence rates using complicated Lyapunov arguments which are difficult to find. To accomplish this, we introduce for convex functions, an asymptotic convergence framework built on certain generalised assumptions which allow for proving the bounded nature of the terms of the Lyapunov function, by exploiting the bounded nature of the globally asymptotic system's solutions. We have shown numerical convergence results for this method, verified using the Whiplash inertial gradient system both in scalar and vector settings for quadratic cost functions.

In this paper, we add to the combined literature presented in the earlier papers and integrate them cohesively geared towards the numerical analysis methods and the novel relaxation sequence methods. Beyond that the contributions of this paper are:
\begin{enumerate}
    \item Numerical results involving near saddle point behaviors and numerical convergence for multidimensional functions have been introduced in \S \ref{num}, adding to the previous numerical results\footnote{Unless mentioned otherwise, the x-axis of all graphs denotes time in seconds. All numerical experiments can be found \href{https://github.com/SubhransuSekharBhattacharjee-01/Whiplash.git}{\color{blue}{here}}.} already introduced for convex costs and Rosenbrock's function.
    \item In addition to the previously introduced results, we have also introduced multiple new lemmas (\ref{jp}, \ref{grab}, \ref{9}) and theorems (\ref{t2}, \ref{t4}, \ref{t5}) that explain the concept of relaxation sequences (in \S \ref{relax}), asymptotic loose bound (in \S \ref{alb}) and numerical rates (\S \ref{ver}), as well as justifications to the results obtained which were not found in the previous papers.
    \item Finally, we have introduced the Whiplash exploration algorithm in  \S \ref{explore}, which uses a momentum sequence based stopping criterion and uses a heuristic termination and restart scheme to drive out of saddle points by exploiting both the first and zeroth order oracles.
\end{enumerate}  
 
\subsection{Organisation}
This paper has been structured into five more sections going forward. Section \ref{sec2} discusses the preliminaries and notations necessary to understand the work, including basic properties and prerequisites in the field of convex analysis. In Section \ref{sec3}, we have introduced the Whiplash gradient descent method - the formulation, discretization, the algorithm, and numerical results in a variety of scenarios. We also introduce the novel Whiplash exploration algorithm motivated from the numerical results and insights obtained from the behaviour of naive methods. This algorithm is heuristic and deterministic and escapes saddle points. In Section \ref{sec4}, we introduce the convergence analyses, including various theoretical results to showcase the convergence rates. We  also introduce our novel relaxation sequences for $L$-smooth functions and their characteristics like the absolute convergence. In Section \ref{sec5}, we introduce the numerical convergence analysis (envelope method) of the closed-loop gradient flow systems by introducing an asymptotic loose bound (which acts as an extension of the $\mathcal{O}$ notation), an integral anchor method and the Lyapunov Rate method (already introduced in \cite{ACC}). We provide a verification of our analysis for the Whiplash gradient descent dynamics using this methodology. In Section \ref{sec6}, we conclude our exposition by discussing the shortcomings of our approach, and some areas of future work in the field of global optimization using the control framework.
 
\section{Background}\label{sec2}
\subsection{Preliminaries}
A central aspect of analyzing gradient-based learning methods is to make necessary assumptions for a class of problems. One such assumption is the Lipschitz continuity of the gradient of the cost function, where there exists a constant $L$, termed as the Lipschitz constant, such that:
\begin{equation}\label{2}
  \| \nabla{f}(y)-\nabla{f}(x)\|\leq\, L \| y-x \|\quad\forall \; x,y\,\in \mathbb{R}^d,
\end{equation}
where $\|.\|$ indicates the $\ell_2$ or the Euclidean norm. We assume this condition for our cost functions throughout our analysis unless specified otherwise. When we discuss iterative gradient schemes in optimization, a sufficient condition for learning the gradient of such a cost function is to take a small enough step-size $s$ such that:
\begin{equation}
    0 < s\leq \frac{1}{L}.\label{3}
\end{equation}
We assume the cost function $f(x)$ is such that $f: \mathbb{R}^{d} \mapsto \mathbb{R}$ is a twice continuously differentiable convex function such that for all $\epsilon \in[0,1]$ \cite{boyd}:
\begin{equation} 
    f(\epsilon x+(1-\epsilon)y) \leq \epsilon f(x)+(1-\epsilon)f(y).
\end{equation}
This relation implies that Jensen's inequality holds
\begin{equation}
\label{jen}
   f(y) \geq f(x)+\langle\nabla f(x), y-x\rangle.
\end{equation}
Rearranging (\ref{jen}), we have the global under-estimator of the convex function $f(x)$ where both $x,y\in\mathbb{R}$ \cite{beck}, we have:
    \begin{equation}
    \label{bomb}
    f(x)-f(y)\leq \langle \nabla f(x), x-y \rangle\quad.
\end{equation}
For $L$-smooth convex functions (which have Lipschitz continuous gradient), where we choose the step-size as $s=\frac{1}{L}$, $\forall\;x,y\in\mathbb{R}$, for $|x^*|< +\infty$ we have
 \begin{gather}
f(y) \leq f(x)+\langle\nabla f(x), y-x\rangle+\frac{L}{2}\|y-x\|^{2}.
\label{ls}
\end{gather} 
We will additionally consider the Polyak-{\L}ojasiewicz (P{\L}) inequality, which implies weak invexity \cite{kar}, which is a local criterion of strict convexity. If we prove convergence for the P{\L} condition, it would naturally hold for a stronger convexity criterion. The P{\L} inequality is satisfied if the following holds for some $\lambda>0$ and every $x\in\mathbb{R}$:
\begin{equation}
\label{PL}
    f(x)-f^{*} \leq \frac{1}{2\lambda}\|\nabla f(x)\|^{2}.
\end{equation}
We use the Cauchy-Schwarz and mean value theorems in our study, both of which can be found in preliminary section of \cite{boyd}.
\subsection{Notation}
We represent the set of all real numbers as $\mathbb{R}$ and the positive real sets as $\mathbb{R}^{+} = \{x \in \mathbb{R} \mid x \geq 0\}$ and $\mathbb{R}_{*}^{+} = \{x \in \mathbb{R} \mid x>0\}$. We represent the superior limit as $\varlimsup$. We use the asymptotic Landau notations $\mathcal{O}$ and $o$, which denotes the loose and strict asymptotic upper bounds \cite{cormen}, defined for $t_0\in[0,\infty)$ respectively for the decay function $\nu(t)>0$ as
\begin{equation}
    \varlimsup_{t\rightarrow t_0}|\nu(t)f(t)| = C< \infty\;\text{is denoted as}\; f(t) = \mathcal{O}\left(\frac{1}{\nu(t)}\right),
\end{equation}
where $C\in \mathbb{R}^+_*$ is an arbitrary constant and  
\begin{equation}
    \varlimsup_{t\rightarrow t_0}|\nu(t)f(t)|= 0\;\text{is denoted as}\; f(t) = o\left(\frac{1}{\nu(t)}\right).
\end{equation}
\section{Whiplash Inertial Gradient Descent}\label{sec3}
\subsection{Formulation}
The Whiplash Inertial Gradient descent design was inspired by a physical understanding of the dynamical system, which involves control using its momentum as developed by Su et al. \cite{su2015differential}, where they model the Nesterov scheme as an ODE. Further generalizations as developed by Attouch et al. \cite{ATTOUCH2016}, using a Hessian-driven approach and Wibisono et al. \cite{wibi} for the formulation of a variational approach using time-scaling lead to the simplified inertial gradient dynamical system:
\begin{equation}
    \ddot{X}+\gamma(t)\dot{X}+\nabla{f}(X) = 0.
\end{equation} 
In \cite{attouch2021fast}, closed-loop control is considered using multiple scenarios with a damping constant of the form $\gamma = r\,|\dot{x}|^{p-2}$, where $p$ and $r$ are positive constants (control parameters) and $\dot{x}$ is the velocity. However, the simulation results we obtained for this algorithm were unsatisfactory for minimizing non-convex functions. We made two particular observations:
\begin{enumerate}
\item The damping function did not adequately stabilize the system over long intervals for various values of $p$ and $r$. This inspired us to use the control parameter $r$ as $t$ (time) and $p = 4$.
\item The lack of stability of the system over large intervals of time led us to look at the system as a linear time-variant system and to perform a corresponding pole placement. This required adding a marginal value to the damping to ensure a lower bound on the damping at $t=0$, such that the damping term never disappears, because that leads to lack of attenuation as the system approaches the minima.
\end{enumerate}
The linearized ODE is equivalent to the form:
\begin{equation}\label{8}
\ddot{X}+\gamma\dot{X}+\nabla^2{f}(X)=0,
\end{equation}
where the Hessian $\nabla^2 f(X)$ is of the objective function at the minima $X=X^*$, we thus look at the eigenvalues for the linear time-variant system matrix corresponding to (\ref{8}) to understand its convergence rate motivated by the approach in \cite{tomas2010linear}. For this purpose, we define a new system with state $y$ for the underlying linearised state space system $\dot{x} = A\,x$, which is given as:
\begin{equation}\label{10}
    y = x\,e^{-\eta t}, \;\; \eta>0\,.
\end{equation}
This relation leads to the redefined system:
\begin{equation}\label{11}
     \dot{y} + \eta\,y = A\,y\,,
\end{equation}
which in matrix form can be written as:
\begin{equation}\label{12}
  \dot{y} = (A\,-\eta I)\,y.
\end{equation}
The minimum real part of the eigenvalues of the system (14) can be written as:
\begin{equation}\label{13}
  \sigma = \frac{-\,\gamma}{2}-\eta,
\end{equation}
which indicates an exponential convergence rate of at least $\eta$. Finally, we redefine the damping $\gamma = K+t\|\dot{x}\|^2$. We choose the term $K = 1$ for the damping coefficient. Note that the choice of $K =1$ is because using $K>1$ leads to unwanted damping in implementation. Our motivation for the choice of $K=1$ was inspired by the Heavy ball dynamical system, as in at $t=0$, the system essentially behaves as a Heavy Ball system. As we know from \cite{hal}, for values $K<1$, the convergence rate is slower because of the unwanted oscillations due to an unstable pole at $t=0$. Hence, the value of $K$ acts as a marginal value for the safe placement of the pole. Thus, we arrive at our hybrid gradient descent optimization method, which we shall refer to as the \textit{whiplash inertial gradient optimization method}:
\begin{equation}\label{wp}
    \ddot{X}+(1+t\,||\dot{X}||^2\,)\dot{X}+\nabla f(X) = 0.
\end{equation}
A block diagram for this system is shown in Figure: \ref{fig:gamma}. The damping of the system for convex costs is as shown in figure: \ref{fig:moment}, with respect to time\footnote{The unique graph of the damping led to the nomenclature of the system as \say{Whiplash}.}.
\begin{figure} 
\centering
    \includegraphics[width = 0.6\columnwidth]{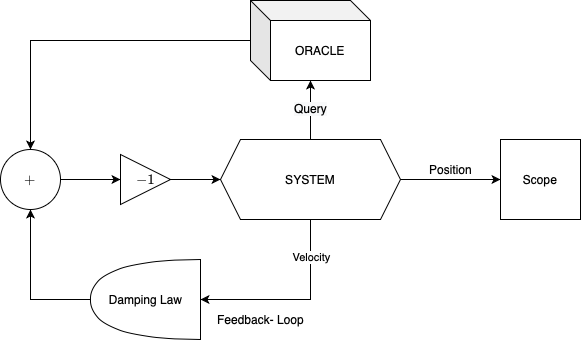}
\caption{Whiplash control system block diagram.} \label{fig:gamma}
\end{figure}
\begin{figure} 
\centering
    \includegraphics[width = 0.65\columnwidth]{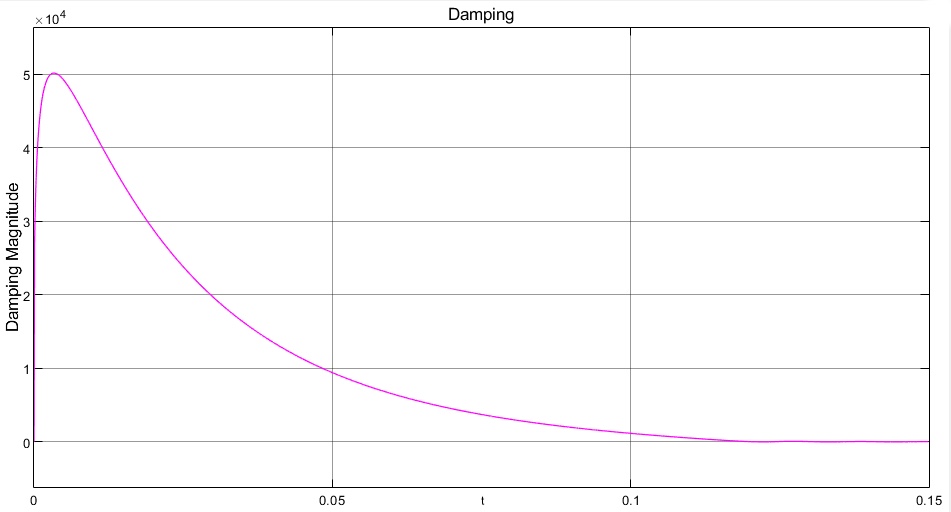}
\caption{Damping of the dynamics with respect to time for convex costs.} \label{fig:moment}
\end{figure}
\subsection{Discretisation \& Physical Model}
\begin{figure} 
    \centering
    \includegraphics[width = 0.75\columnwidth]{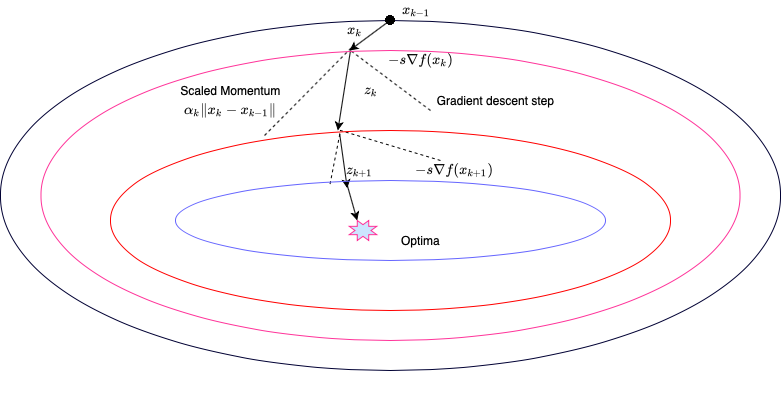}
    \caption{Description of the whiplash gradient descent method.}
    \label{fig:scheme}
\end{figure}
The discretisation used is the semi-implicit or symplectic Euler method \cite{bp}. Using a discrete-time step $s$ and sampling period $t \approx \, k\sqrt s$, we obtain a two-state recursive estimate of the acceleration and velocity as shown below as the equivalent of the Hamiltonian flow for the state-space, given as:
\begin{equation}\label{20}
\begin{split}
& \dot{X} \approx v_k = \frac{x_{k+1}-x_{k}}{\sqrt{s}},\\[0.1cm]
& \ddot{X} \approx \frac{v_{k+1}-v_{k}}{\sqrt{s}},\\
& 1+t\,||\dot{X}||^2\, \approx 1+k\sqrt{s}||v_k||^2.
\end{split}
\end{equation}
Now, we modify (\ref{wp}) to add a fixed mass to the system, which has been considered unit magnitude up until this point. The choice of mass that we shall make is $m = \frac{1}{\sqrt{s}}$. The idea of introducing this mass in inertial gradient flow methods while discretizing them has been inspired by selective mass scaling in finite element methods \cite{fem}, where the iterative process can be scaled by choosing an effective mass. The rationale is that since the discrete-time method depends heavily on the step size, it will take much longer to converge for smaller step sizes. Hence, to counter this effect, we introduce a fixed mass, which scales the dynamics depending on the step size. Upon making these substitutions and modifications to (\ref{wp}), as a model of the second order damped oscillator, with no external forces or perturbation, and a potential function $g(x)$, mass $m>0$ and damping $c>0$, given as
\begin{equation}
    m\ddot{x}(t)+c\dot{x}(t)+g\left(x(t)\right)=0,
\end{equation}
we obtain
\begin{equation}\label{21}
    \frac{1}{\sqrt{s}}\cdot \frac{v_{k+1}-v_{k}}{\sqrt{s}} + (1 + k\sqrt{s}||v_k||^2)\frac{v_k}{\sqrt{s}} +\nabla f(x_k) = 0.
\end{equation}
We can re-write (\ref{21}) using (\ref{20}) as:
\begin{equation}
\label{22}
v_{k+1} = (1-\sqrt s - ks\;v_k^T \cdot v_k)v_k - s\nabla{f(x_{k})}\,.
\end{equation}
Now, we consider the symplectic approximation for the Lyapunov stable system \cite{hair} such that $\|\dot{x}(t)\|\rightarrow0$ as $t\rightarrow \infty$. This implies that $\|v_k\|\rightarrow0$ as $k\rightarrow\infty$. Therefore, we introduce $z_k = x_{k}-x_{k-1} = \sqrt{s}v_{k-1}$. For asymptotic analyses, there is no difference between the sequences $z_k$ and $ v_k$ asymptotically as
\begin{equation}
    \lim_{k\rightarrow \infty} \|z_{k+1}-v_k\| = \lim_{k\rightarrow \infty}\,|\sqrt{s}-1|\|v_k\| = 0.
\end{equation}
We may consider this as two transforms. First as a scaling of the system, followed by a backward recursion:
\begin{equation}
    \lim_{t\rightarrow\infty} |\sqrt{s}-1| \dot{x}(t)
\approx \max_{0<k\leq \frac{T}{\sqrt{s}}} (x(k\sqrt{s})-x((k-1)\sqrt{s})).
\end{equation}
This trick simplifies our system's updates while keeping intact the geometry of the dynamical system and does not change the global nature of the system's convergence. This particular design choice for the algorithm simplifies the computation and makes discrete-time analyses of convergence considerably easier. We finally have the consolidated scheme as follows:
\begin{equation}
\label{WGDA}
\hat{\mathcal{W}} = 
    \begin{cases}
    & x_1 = x_0 - s \nabla{f(x_0)},\\
    & z_k = x_{k}-x_{k-1},\\
    & \alpha_k = 1-\sqrt s - ks\|z_k\|^2, \\
    & z_{k+1} = \alpha_k\,z_k - s\nabla f(x_{k}),
    \end{cases}
\end{equation}
where condition (\ref{3}) holds. The whiplash method generates the momentum of the cost function to scale its next move in the discrete iteration. Figure: \ref{fig:scheme} provides a representation of the algorithm of (\ref{WGDA}).
\begin{figure}
    \centering
    \includegraphics[width = 0.3\columnwidth]{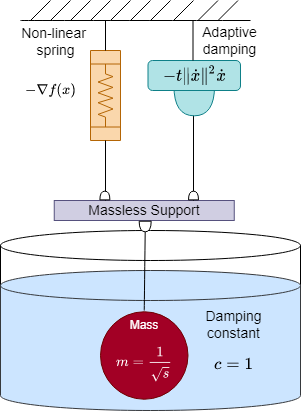}
    \caption{Physical model of the whiplash gradient descent optimization method.}
    \label{fig:model}
\end{figure} 
We can model the whiplash gradient descent as a physical system, consisting of mass $m= \frac{1}{\sqrt{s}}$, where $s>0$ (is the step-size of the algorithm), in a fluid of damping constant $c=1$ as shown in Figure: \ref{fig:model}. We can model the oracle as a spring of varying restoring force $-\nabla f(x)$ and the non-linear damping as an adaptive damper, connected by rigid inextensible links.
\begin{figure}
    \centering
    \includegraphics[width=0.35\columnwidth]{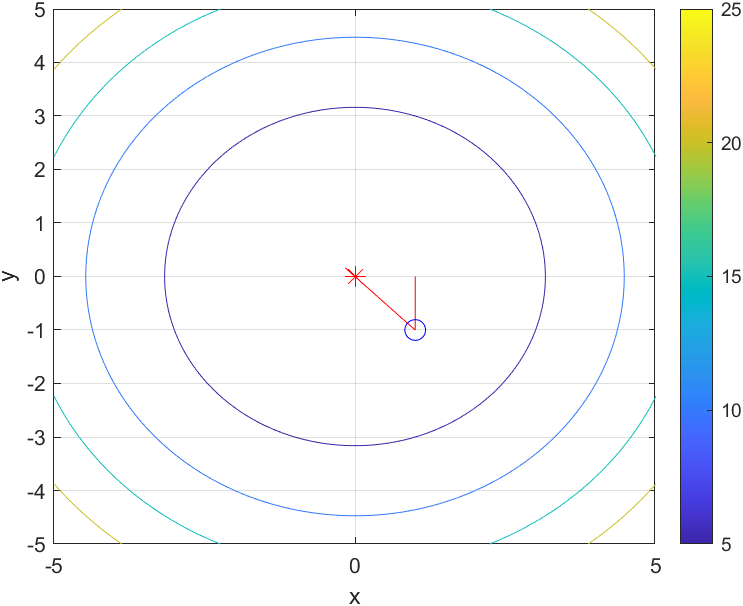}
    \includegraphics[width=0.35\columnwidth]{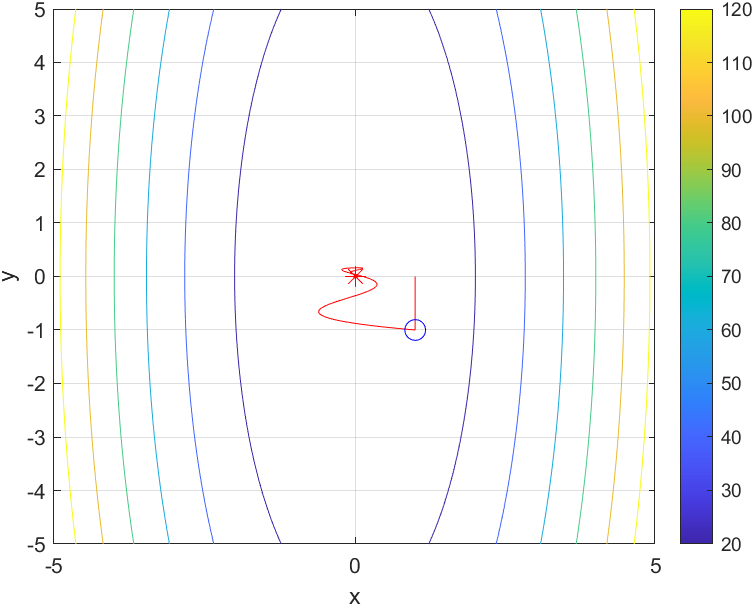}
    \includegraphics[width=0.35\columnwidth]{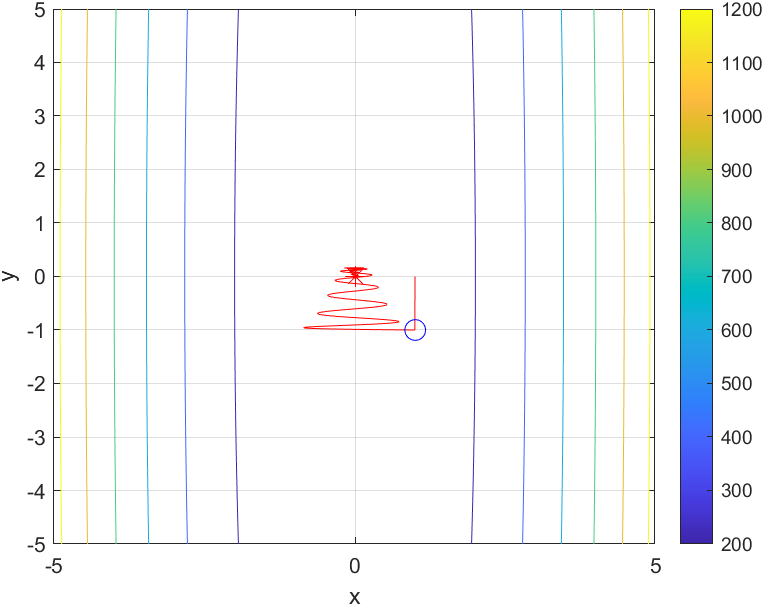}
    \includegraphics[width=0.35\columnwidth]{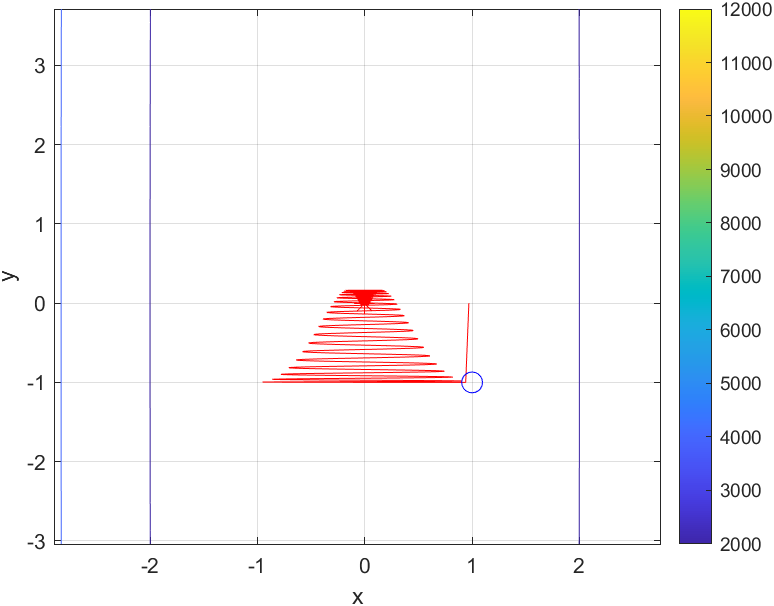}
    \caption{A study of performance of the whiplash gradient descent for increasing condition numbers: $\kappa =1$ (top left), $\kappa =10$ (top right), $\kappa =100$ (bottom left), $\kappa =1000$ (bottom right).}
    \label{fig:stud}
\end{figure}
\subsection{Naive \& Momentum-based stopping algorithms}
This discrete-time scheme (\ref{WGDA}) can be translated to the following algorithm (\ref{alg:1}) using a step-size $s$ and $n$ iterations, with initial starting point $x_0$ and final point $x_n$. Unlike classical gradient descent algorithms, this algorithm does not use any hyper-parameters. Instead, it uses a simple two-step assignment to update the discrete-time damping in every iteration. The \textit{zeroth} step of the iteration $x_1$ is a gradient descent step \cite{boyd} which assigns the initial momentum for the first iteration.
\begin{algorithm} 
\caption{Naive Whiplash algorithm}
\label{alg:1}
\SetKwInOut{init}{Initialise}
\SetKwInOut{Input}{Input}
\SetKwInOut{oracle}{Oracle}
\SetKwInOut{out}{Output}
\SetKwBlock{block}{Program}
\DontPrintSemicolon
%algorithm
\Input{$x_0, n, s$}
\KwData{$\nabla{f}$ ($1^{\text{st}}$ Order Oracle)}
\KwResult{$x_k$}
\init{$x_1 \leftarrow x_0 - s \nabla{f(x_0)}, k = 1$}
\block{
\While{$k \leq n$}{
    $z_k \leftarrow x_{k}-x_{k-1}$\;
    $\alpha_k \leftarrow 1- \sqrt s - ks(z_k^T \cdot z_k)$\;
    $x_{k+1} \leftarrow x_{k} +\alpha_k z_k - s\nabla{f(x_{k}) }$\;
    $(x_{k-1},x_{k}) \leftarrow (x_{k},x_{k+1})$\;
    $k \leftarrow k+1$\;
    }
\Return{$x_k$}\;
}
\end{algorithm}
A modified version of algorithm (\ref{alg:1}) is constructed in algorithm (\ref{alg:2}) on the basis of the momentum criterion. We introduce a small stopping value $\epsilon>0$, which allows for early termination of the algorithm, significantly reducing computation for convex functions. Note that in this pseudo-code, we have $x$ as the initial vector, $t$ as the output vector and $y$ as an intermediate variable. We discuss the motivation in Remark: \ref{z}.
\begin{algorithm} 
\caption{Whiplash algorithm with Momentum-based stopping criterion}
\label{alg:2}
\SetKwInOut{init}{Initialise}
\SetKwInOut{oracle}{Oracle}
\SetKwInOut{out}{Output}
\SetKwInOut{Input}{Input}
\SetKwBlock{block}{Program}
\DontPrintSemicolon
%algorithm
\Input{$x, s, \epsilon$}
\KwData{$\nabla{f}$ ($1^{\text{st}}$ Order Oracle)}
\KwResult{$t$}
\init{$t = x - s \nabla{f(x)},\;k = 1,\;z = t-x,\;p=z^T z$}
\block{
\While{$p\leq\epsilon$}{
    $\alpha \leftarrow 1- \sqrt s - k\cdot s\cdot p$\;
    $y \leftarrow t +\alpha \cdot z - s\nabla{f(t)}$\;
    $(x,t) \leftarrow (t,y)$\;
    $k \leftarrow k+1$\;
    $z \leftarrow t-x$\;
    $p \leftarrow z^T\cdot z$\;
    }
\Return{$t$}\;}
\end{algorithm}
However, in the absence of strictly convex criterion, the algorithm (\ref{alg:2}), would not terminate and instead go on to an infinite loop. Therefore, upon encountering saddle points, as explained in Remark: \ref{pie}, we need to use additional stopping criterion, which are based on heuristic evaluations of the class of functions under study. We discuss such a novel algorithm using the Whiplash scheme in \S \ref{explore}.
\subsection{Numerical Results}\label{num}
\begin{enumerate}
\item As explained previously, for the Rosenbrock function, for the stiffness of the system, we use a step-size of no more than $10^{-5}$. This is because the algorithm is unable to learn the gradient of the cost function and picks up momentum without correcting the damping. For a sufficiently small step-size, the whiplash gradient descent algorithm successfully found the minima of Rosenbrock's function for all initial conditions. We have shown a few examples in Figure \ref{fig:7}. This indicates that Rosenbrock's function optimization has been achieved over the given time interval. Figure \ref{fig:9} shows the nature of the trajectory, as it approaches the minima of the Rosenbrock function.

\item We analyse the performance of our algorithm for convex cost functions. The standard method of determining the algorithm's effectiveness is analysing its performance for various geometries. Note that, for unconstrained minimisation, this is a benchmark test because it allows us to analyse the tractability of the algorithm over a wide range of geometries. We perform a few experiments to test the algorithmic performance, as shown in  Figure \ref{fig:stud} using the function $f(x, y) = \frac{1}{2}(x^2+\kappa y^2)$ for the starting point $(1,-1)$, where the condition number $\kappa$ is varied.
\item In \cite{Whiplash}, we considered the utility of non-zero starting velocities for non-convex functions. We observed that, for non-convex or ill-conditioned systems, the initial velocity direction often dictates the nature of convergence. For the convex case, we observe that the direction of the starting velocity is immaterial to the nature of convergence, as shown in Figure: \ref{fig:vel}. As shown in Figure: \ref{fig:vel2}, for small starting velocities, i.e. $|v|<0.1$, the nature of the trajectory is independent of the direction of the starting velocity.
\item We made a few observations for the numerical experimentation of the naive Whiplash algorithm around Saddle points. We observed that as long as the starting point is not at the saddle point itself, it is able to escape the saddle point. It is observed for the function $J =\frac{1}{2}(x^2-y^2)$, which has a saddle point at $(0,0)$ we observed that if the algorithm is started sufficiently far from the saddle point, it does find the saddle point but escapes it. The numerical results for the saddle point experiment has been shown in figure: \ref{fig:sad}. 
\end{enumerate}
\begin{remark}\label{pie}
The disadvantage of our naive algorithm is for functions with saddle points, it is unable to terminate because the approach being fundamentally black box prevents it from storing global information which means it may oscillate around a similar neighbourhood as it has no memory of the neighbourhoods visited in the past nor can it restart automatically from a different neighborhood. To counter this issue, we introduce a restarting scheme which force the algorithms out of neighbourhoods of the saddle points, as shown in algorithm (\ref{alg:3}), introduced in the next section. The algorithm has been verified for the cost function $J$. We find not finite minima in $r = 20$ restarts, which numerically confirms that our algorithm is successful in deterministically escaping the saddle point.
\end{remark}
\begin{figure} 
        \centering
            \includegraphics[width = \columnwidth]{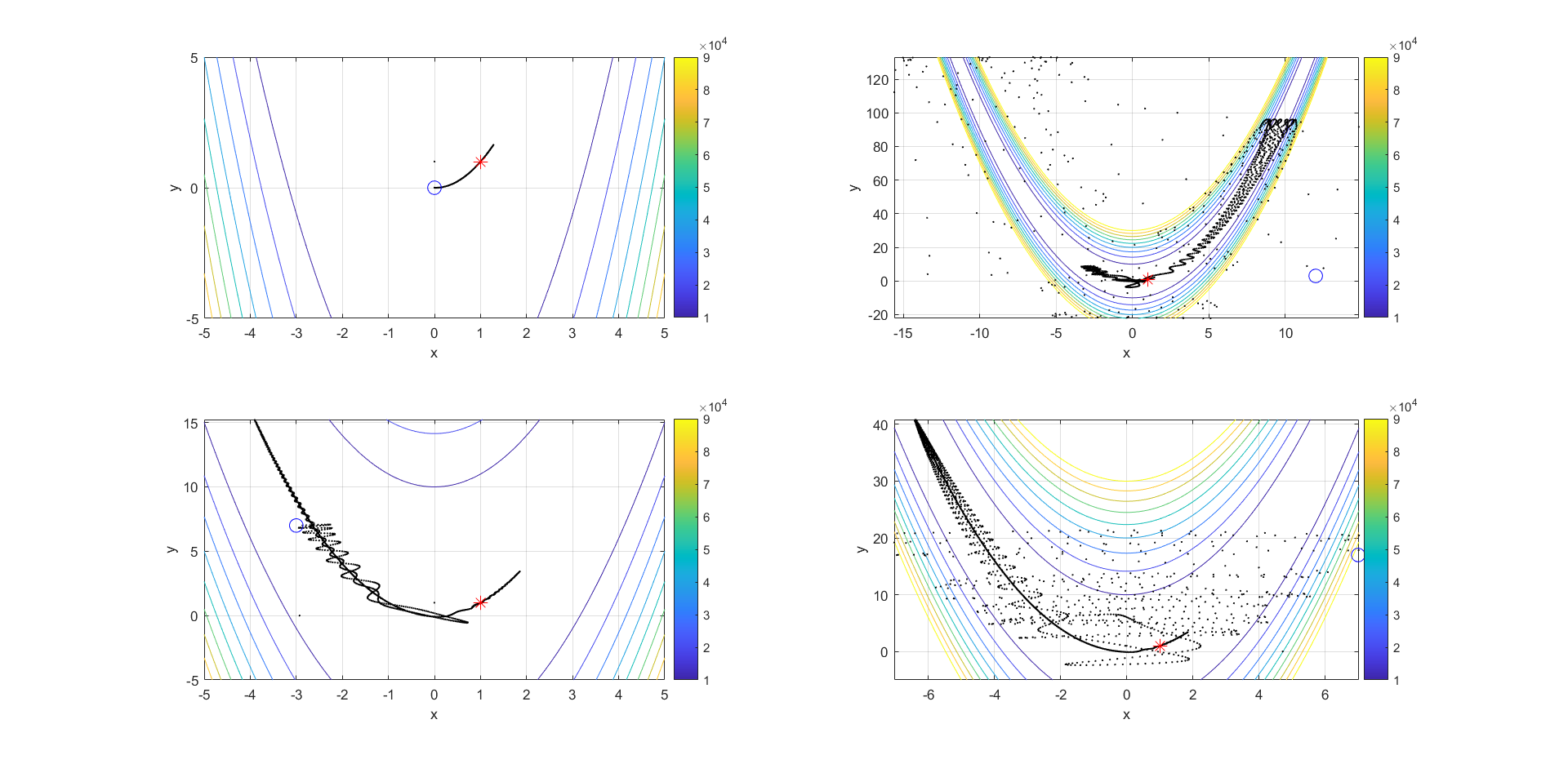}
        \caption
        {Trajectories for the Rosenbrock function optimization are dependent on initial condition: Note the blue circle denotes the starting point and the red star denotes the ending point.
        From upper left corner the initial conditions are (0,0), upper right(12,3), lower left (-3,7) and lower right (-7,17).}
        \label{fig:7}
    \end{figure}
\begin{figure}
    \centering
    \includegraphics[width = 0.75\columnwidth]{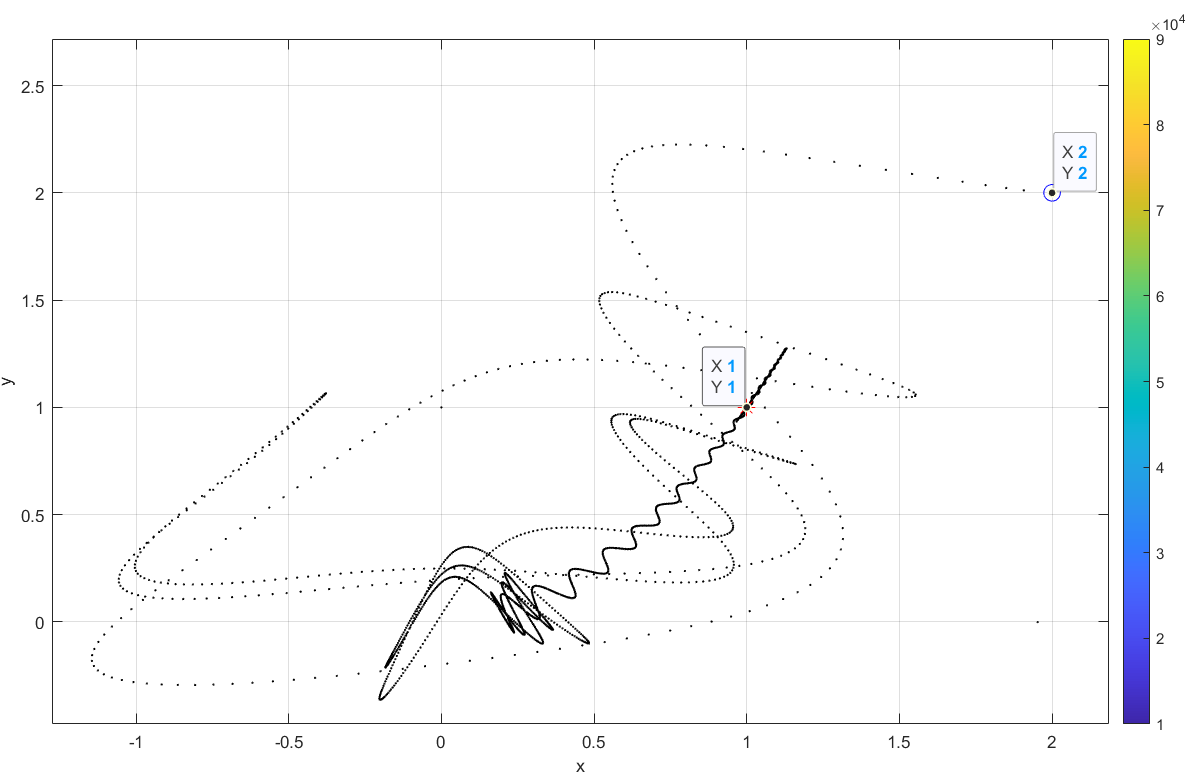}
    \caption{Zoomed Trajectory Plot with starting point (2,2).}
    \label{fig:9}
\end{figure}
\begin{figure}
    \centering
    \includegraphics[width = 0.7\columnwidth]{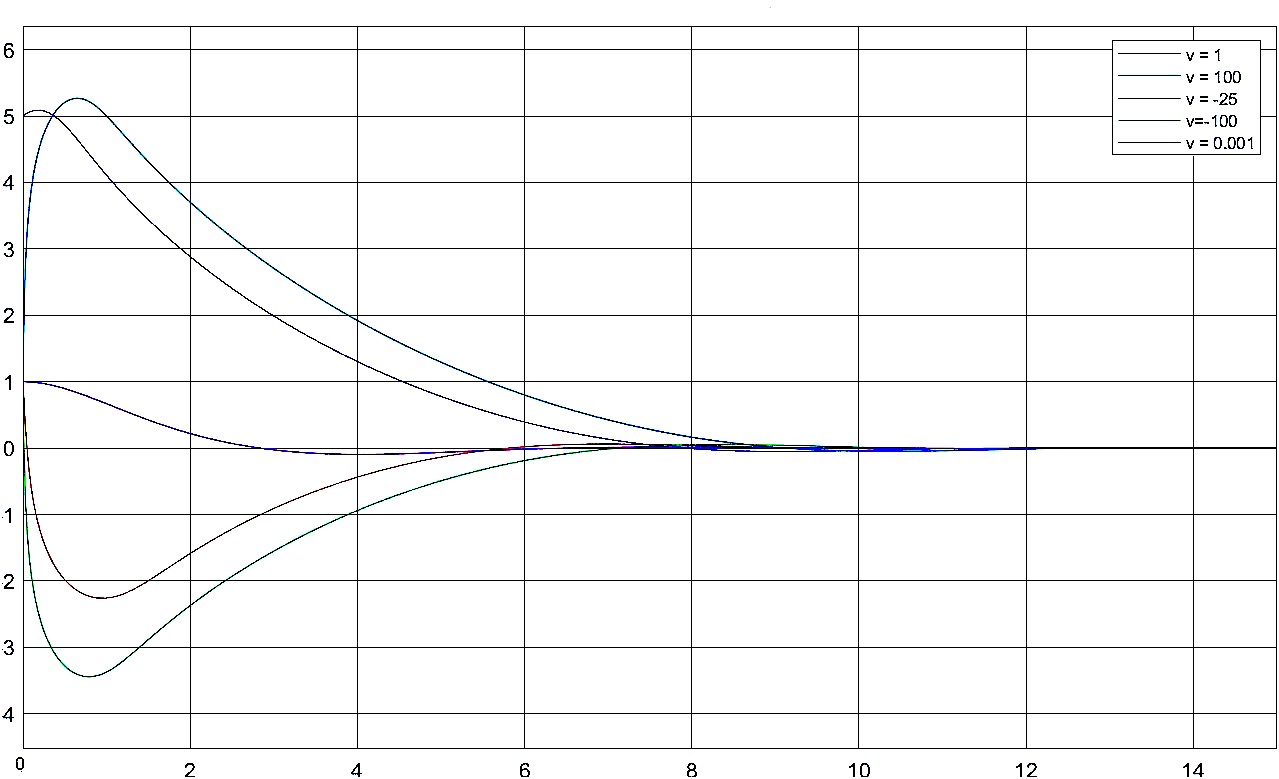}
    \caption{Study of performance of the whiplash method under various starting velocities.}
    \label{fig:vel}
\end{figure}
\begin{figure}
    \centering
    \includegraphics[width = 0.7\columnwidth]{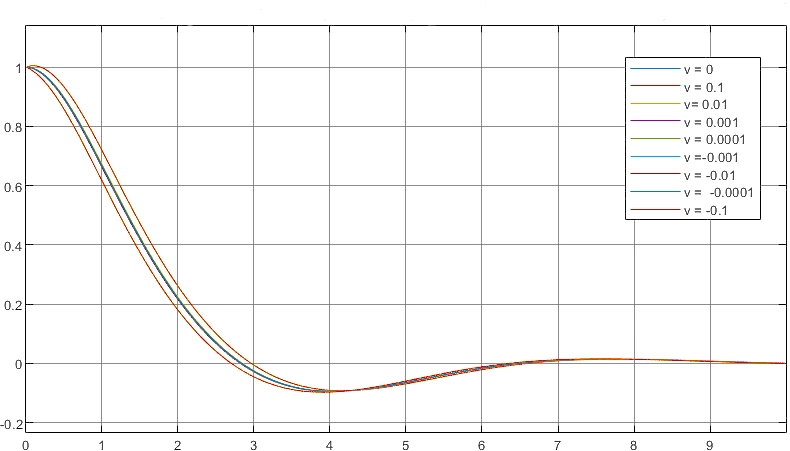}
    \caption{Trajectories for small starting velocities.}
    \label{fig:vel2}
\end{figure}
\begin{figure}
    \centering
    \includegraphics[width = 0.7\columnwidth]{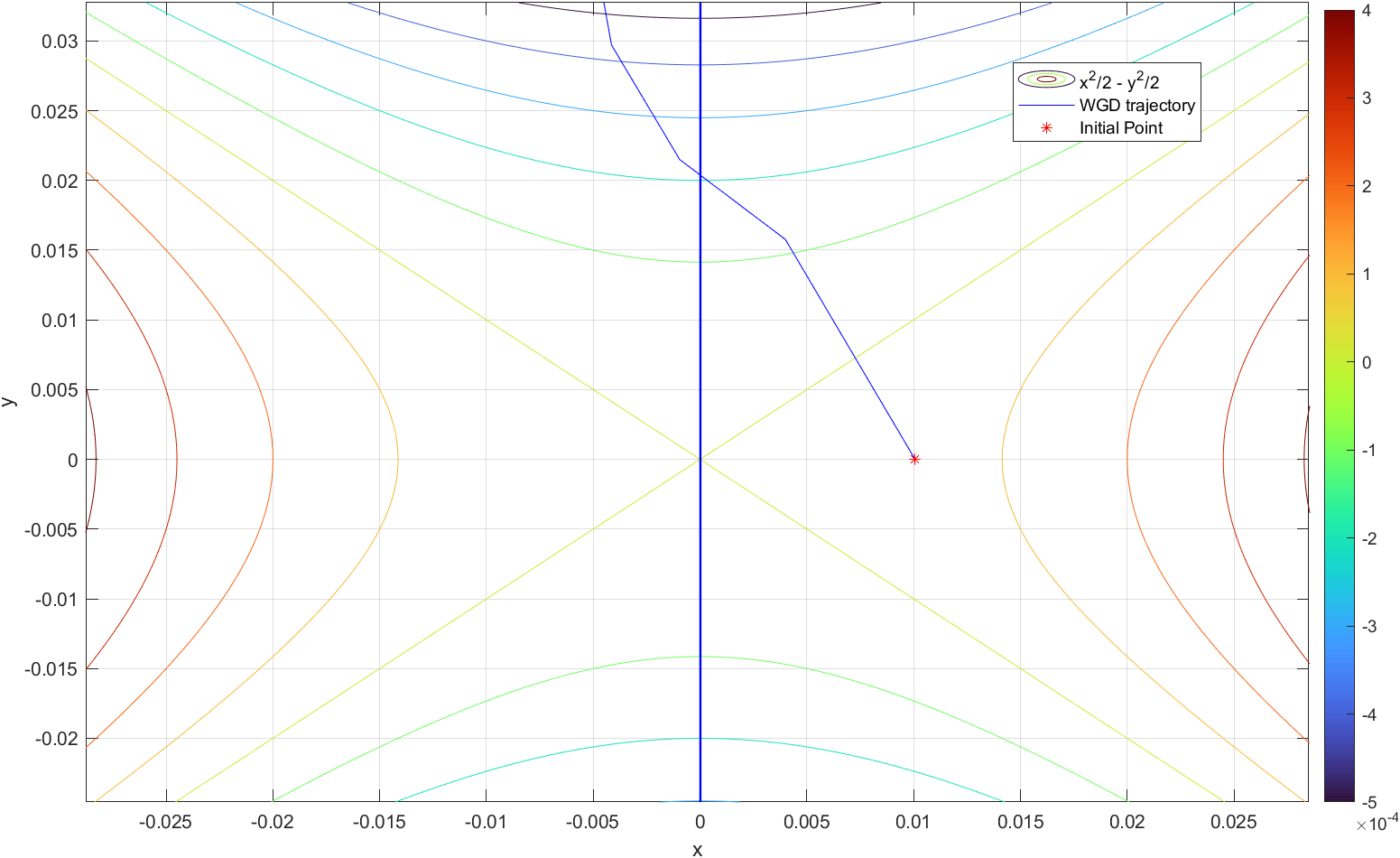}
    \caption{Numerically escaping the saddle point for the function $J = \frac{1}{2}(x^2-y^2)$ given the starting point (-0.01,-0.01).}
    \label{fig:sad}
\end{figure}
\subsection{Whiplash exploration algorithm}\label{explore}
We introduce a final algorithm (\ref{alg:3}) which terminates automatically depending on heuristic parameters. Note that we cannot guarantee finding global minimas for non-convex functions in finite-time. Hence, inspired by Global Optimization methods \cite{global}, we introduce an auto restart variable $r$ which uses the maximal radius explored by the algorithm to find another initial point. After $r$ restarts, upon $N$ maximum iterations on every restart batch, the algorithm auto-terminates. This algorithm specialises at exploring the function locally and deterministically. Hence, we have named this as \emph{Whiplash exploration algorithm}. The functional setting for the cost function is locally convex globally non-convex function with finite number of saddle points and local minimas represented by the finite set $\mathcal{X}$, with one global minima $x^*\in\mathcal{X}$, which maps as $\mathbb{R}^n\mapsto \mathbb{R}$, such that $f^* = \min_{x^*\in \mathcal{X}} f$.
The strategy we follow is as follows:
\begin{enumerate}
    \item We choose the the initial point $x_i$, step-size $s>0$, the maximum number of permitted iterations in a batch $N$, the number of restarts $r\geq0$, the exploration parameter $\lambda>1$ and the stopping criterion $\epsilon \geq 0$.
    \item Initialise the oracle as a data line for the algorithm which allows it to sequence the gradient values in a black-box format. We also use the function output as a zero-order oracle.
    \item Initialise counter variables as zero, while retaining the other initialisation and the initial direction vector $$u=[1 \dots n\; times \dots 1]^T = I^T_n.$$
    \item Initialise momentum using the gradient descent for the initial step as in the naive algorithm.
    \item Initialise the reference radius $R_r = \|t-x_i\|$ which allows the algorithm to decide whether it has started at a stationary point. Note that we are performing a trade-off. Essentially we are trading off faster convergence for exploration. The initial radius of exploration is the reference radius.
    \item If the algorithm has not restarted, we enter the iteration loop for $N$ as $k$ updates and performs the naive whiplash algorithm scheme. If $k>N$, then the algorithm restarts.
    \item The exploration radius $R_k$ is updated if it is higher than the reference radius (in the first iteration) or since the last iteration. This allows for finding the maximum radius in the batch.
    \item As the algorithm iterates, if it finds a local optima, it moves towards the minima, accurate to an error magnitude of $\epsilon$. In case the momentum does not converge, in the presence of a saddle point or ill-conditioned plains, we iterate the algorithm up to a tolerating limit $N$, after which we restart the algorithm from a different initial point.
    \item In case the momentum has converged, if the algorithm has searched sufficiently, i.e. $R>\lambda R_r$, the algorithm breaks out of the loop and finishes to return the value of the point, otherwise it restarts, as the batch must have been started near a saddle point. The choice of exploration parameter, which is the scaling magnitude $\lambda$, is determined heuristically. In the exceptional case where the algorithm restarts multiple times near local minimas, the lowest value $f(y)$, where $y\in A$ is the output. To perform this we save the value in an array $A$.
    \item To restart in the next batch if the minima is not found within the search radius, we employ the following restarting scheme:
    $$x_0 = x_i+2R\cdot \frac{\nabla^T{f}(t)\cdot u}{\|\nabla{f}(t)\|},$$ whereby we can generate a new search space at every batch using an input unit vector $u$, which can be assigned arbitrarily. Essentially, we want to prevent the search in the same $\mathbb{R}^n$ space covered in the first batch and hence, we restart algorithm at double the maximum explored radius $R$ of the initial batch. This process is visualized in figure:\ref{fig:hyper}.
    \item  The value of $u$ can be changed at every iteration, however, in our implementation, it is updated deterministically, and is given as $$u' = u+\frac{u+\left|\frac{d}{1-d}\right|^d}{\left\|u+\left|\frac{d}{1-d}\right|^d \right\|}.$$
    \item Finally, as a terminating criterion, if the algorithm has explored heuristically for the minima and has either found the minima or has exhausted restarts it makes a decision on the output. If the restarts have been exhausted the algorithm auto-terminates. If the position vector $t=\infty$ while searching for a minima, the algorithm restarts, thus speeding up the search process. In implementation we make sure that the value is not \texttt{NaN}. Hence, only if the momentum sequence has converged and the algorithm has explored sufficiently far enough, determined by the choice $\lambda$, the algorithm produces a viable output.
\end{enumerate}
\begin{algorithm}
\caption{Whiplash exploring algorithm (with Momentum-based stopping criterion and Heuristic saddle point restart)}
\label{alg:3}
\SetKwInOut{ini}{Initialise}
\SetKwInOut{oracle}{Oracle}
\SetKwInOut{out}{Output and Exit Program}
\SetKwInOut{Input}{Input}
\SetKwBlock{block}{Program}
\DontPrintSemicolon
\textbf{Input:} $x_i, s, \lambda, N, r, \epsilon$\\
\KwData{$f$ ($0^\text{th}$ Order Oracle), $\nabla{f}$ ($1^{\text{st}}$ Order Oracle)}
\KwResult{$t$}
\textbf{Initialise:} $d=k=j=0$, $u = I^T_{n}$, $t = x_i - s \nabla{f(x_i)}$, $z = t-x_i$, $p=\sqrt{z^T z}$, $R=R_r = \|t-x_i\|$ \& array $A[N]$\\
\block{
restart:\\
\If{$d >0$}
{ 
    \eIf{$d>r$}{\textbf{go to} fail}{
    $x_i \leftarrow x_i+2R\cdot \frac{\nabla^T{f}(t)\cdot u}{\|\nabla{f}(t)\|}$; 
    $u\leftarrow \frac{u+\left|\frac{d}{1-d}\right|^d}{\left\|u+\left|\frac{d}{1-d}\right|^d \right\|}$\;
    $k\leftarrow 0$\;
    }
}
\While{$k<N$}{$x\leftarrow x_i$\;
\While{true}{
    $k \leftarrow k+1$\;
    \If{$k>N$ or $\|t\|=\infty$}{$d \leftarrow d+1$\;
    \textbf{go to} restart\;}
    $\alpha \leftarrow 1- \sqrt s - k\cdot s\cdot p^2$\;
    $y \leftarrow t +\alpha \cdot z - s\nabla{f(t)}$\;
    $(x,t) \leftarrow (t,y)$\;
    $z \leftarrow t-x$\;
    $p \leftarrow \sqrt{z^T\cdot z}$\;
     $R_k \leftarrow \sqrt{(t-x_i)^T(t-x_i)}$\;
    \If{$R_k>R$}{$R=R_k$\;}
    \If{$p\leq \epsilon$ }{
    \eIf{$R\geq\lambda\cdot R_r$}{\textbf{go to} finish\;}{$A[j] = t$\;$j\leftarrow j+1$\;
    \textbf{go to} restart\;}}
    }
    }
fail: \out{No minima found in finite time for the given heuristics.}
finish:\\ \eIf{$A = null$}{\Return{$t$\;}}{$j\leftarrow 0$\;
\While{$A[j]$}
{$A[j] \leftarrow f(A[j])$\;
$j\leftarrow j+1$\;}
$t\leftarrow \min \{ \min \{A\}, f(t)\}$\;
\Return{$t$\;}
}}
\end{algorithm}
\begin{figure}
    \centering
    \includegraphics[width = 0.4\columnwidth]{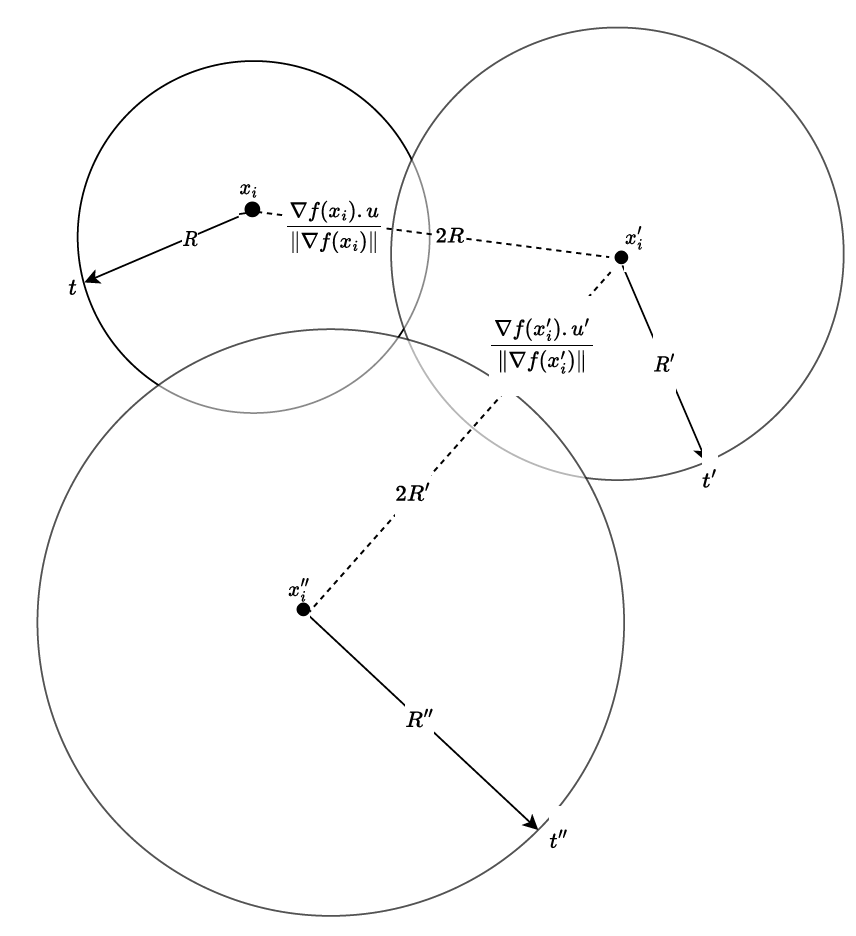}
    \caption{Evolution of the maximum radius and the restart of the Whiplash exploring algorithm with respect to the gradient descent direction.}
    \label{fig:hyper}
\end{figure}
Note that in our implementation, we make the following set of choices: $s=10^{-4}$, $\lambda=2$, $r=20$, $N=100000$ and $\epsilon=10^{-6}$. We implement this algorithm for the class $f = \frac{1}{2}\left(x^2+\kappa y^2\right)$, for various $\kappa>0$. Note that our algorithm does not adapt input heuristics. Stochastic approaches for designing the heuristic can be used in future work \cite{primer}. For example the direction vector $u$ could be a random multivariate Gaussian vector, normalised to unit magnitude. Other computational approaches like parallel-processing algorithms could be used to further speed-up the search process by exploiting local memory (cache) options within the system. 

It is important to address the fact that in case the algorithm runs out of restarts and fails to find the global minima, it produces the minimal vector of a discovered stationary point. This limitation, which might produce sub-optimal results, is heuristic-driven. Our algorithm is strictly limited to gradient information in its run-time and, therefore, cannot adapt to the geometry of the function itself nor can it verify whether the minima found is global.

Finally, it is important to address the fact that there is a probability that every restart might end up spending iterations in neighbourhood of points it has already explored. To prevent this one could always update the input $u=\frac{\nabla{f}(t)}{\|\nabla{f}(t)\|}$. However, numerically we have not found any significant difference in results. It is useful to thus consider heuristics using priori learning or systematic knowledge of the local geometries.
\section{Convergence analysis for \emph{L}-smooth convex cost functions}\label{sec4}
\subsection{Symplectic Asymptotic Convergence}
We consider the following lemmas---\ref{1}, \ref{k} and \ref{m}, which bridge from the continuous-time analysis of convergence to the discrete-time analysis using the symplectic discretisation property of asymptotic behavioural invariance. We analyse the Whiplash inertial gradient dynamical system which originates using an inverted approach of Wibisono et al.'s \cite{wibi} and Su et al.'s \cite{su2015differential} approach towards integrating discrete dynamical systems. The step-size $s$ is considered to be sufficiently small in our analysis. Instead of using conventional methods of proving sequence based convergence estimates, using the asymptotic stability of the system itself, we essentially exploit the idea \say{rate-preserving} symplectic discretisation which allows us to use the similar arguments of convergence analysis in the discrete domain. Our methodology for proving convergence exploiting Lyapunov stability theorem and symplectic discretisation instead of finding estimate sequences \cite{baes} or other discrete methods is summarised in figure: \ref{fig:g}. Thus, we elucidate a non-classical methodology of proving convergence using an indirect approach, named as \emph{symplectic asymptotic convergence} method. 
\begin{figure}  
    \centering
    \includegraphics[width = 0.65\columnwidth]{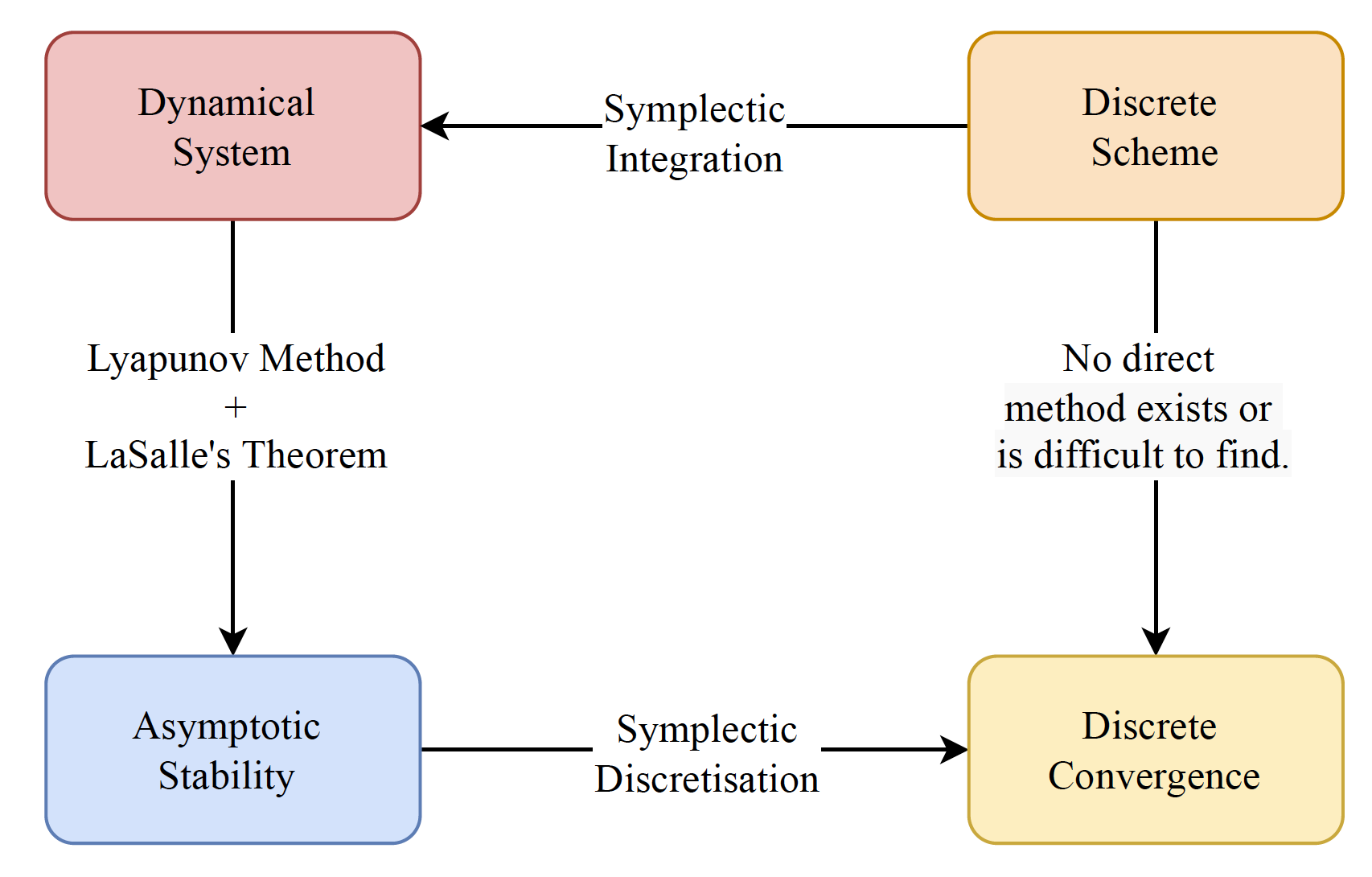}
    \caption{Symplectic Asymptotic Convergence analysis methodology}
    \label{fig:g}
\end{figure}
\begin{lemma}\label{1}
The Whiplash inertial gradient dynamical system (\ref{wp}) is globally asymptotically stable.
\end{lemma}
\begin{proof}
For an autonomous system of the form:
\begin{equation}\label{15}
    \dot{x} = f(x), 
\end{equation}
we can guarantee global asymptotic stability if there exists a functional, $V\,(x)$, such that:
\begin{equation} \label{16}
\begin{split}
 &  V(x)>0\;\; \forall \,x\neq 0 \\
 &  V(0)=0 \\
 &  \dot{V}(x)\,<\,0\;\; \forall\,x \neq 0\,.
\end{split}
\end{equation}
We consider the La-Salle principle of invariance \cite{Khalil} and suppose there exists a continuously differentiable, positive definite, radially unbounded function  $V(z)\,: \mathbb{R}^n \rightarrow \mathbb{R}$ such that $\forall \;z \in \mathbb{R}^n$:
\begin{equation}\label{17}
  \frac{\partial{V}}{\partial{x}}(z-x_e)\,f(z)\leq W(z) \leq 0\,.
\end{equation}
Then, $x_e$ is a Lyapunov stable equilibrium point, and the solution always exists globally. \say{Moreover, $x(t)$ converges to the largest invariant set $M$ contained in $E\,=\,\{ z \in \mathrm{R^n} \, :W(z)=0\}$. When $ W(z)=0$  only for $ z=x_{e}$ then $ E=\{x_e\}$. Since $ M \subset E$, therefore, $x(t) \rightarrow x_e $ which implies asymptotic stability. Even when $E \neq \{x_e\}$, we often have the condition  $M=\left\{x_e\right\}$ from which we can conclude asymptotic stability} \cite{Kalman1960}.
This result is used in our analysis for the general inertial gradient dynamical system \cite{ATTOUCH20175412} by defining a candidate Lyapunov function $W(t)$ for all damping functions $\gamma$ such that:
\begin{equation}\label{18}
W(t) = \frac{1}{2} ||\dot{x}||^2 + f(x) - f(x^*),
\end{equation}
where $f(x^*)$ denotes the minima of the cost function. Upon replacing the time derivative in the equation (\ref{18}) for $\ddot{X}$, we obtain:
\begin{equation}\label{19}
   \begin{split}
    &  \dot{W}(t) = \langle \dot{X},\ddot{X} \rangle + \langle \nabla(f),\dot{X}
    \rangle \\
    & = - \; \gamma ||\dot{X}||^2\; \leq 0\ \; \forall \, \gamma \geq 0\,.
    \end{split} 
\end{equation}
This relation shows that the time derivative of our Lyapunov candidate is negative and semi-definite. This condition shall suffice to show using La-Salle's principle of invariance that the set of accumulation points of any trajectory is contained in $\mathcal{I}$, where $c$ is the union of complete trajectories contained entirely in the set $\{\mathbf{x}:\dot {W}(\mathbf{x})=0 \} $\cite{Khalil}. Thus by Lyapunov's Second theorem, we have the functional $W$ is positive definite; i.e. \say{$\mathcal{I}$ contains no trajectory of the system except the trivial trajectory $\mathbf{x}(t) \equiv \mathbf{0}$ and as $W$ is radially unbounded; i.e. $W(\mathbf{x} )\rightarrow \infty$ as $\mathbf{||x||} \rightarrow \infty$}, we conclude that the origin is \textit{globally asymptotically stable} \cite{Kalman1960}. Since from (\ref{wp}), has $\gamma = (1+t\|\dot{x}\|^2)>0$, hence, it suffices the condition to be globally asymptotically stable as required.
\end{proof}

\begin{lemma}\label{k}
For the whiplash gradient descent method, the sequence $\|x_{k}-x^*\|$ converges as $k\rightarrow\infty$ for a convex cost function.
\end{lemma}
\begin{proof}
The system is globally asymptotically stable for the whiplash gradient descent, as shown in Lemma 1. We know that for the symplectic \say{rate preserving} Euler scheme \cite{hair}, the asymptotic nature of the inertial gradient system remains intact within the sampled time-domain $T>0$; i.e.
\begin{equation}
    \max _{0 < k \leq \frac{T}{\sqrt{s}}} \|x_k-x(k\sqrt{s})\| = 0.
\end{equation}
This means that as $\|x(t)-x^*\|$ is asymptotically convergent, $\|x(k \sqrt{s})-x^*\|$ is asymptotically convergent, as $k\rightarrow\infty$. Then it follows that, for a step-size $0<s\leq 1$, the following must hold
\begin{equation}
    \lim_{k\rightarrow\infty} x_k = x^*.
    \label{e1}
\end{equation}
Thus, the sequence $\|x_{k}-x^*\|$ must be convergent, as required.
\end{proof}

\begin{lemma}\label{m}
The momentum sequence $\|z_k\|$ of the whiplash gradient descent method is bounded and convergent.
\end{lemma}
 
\begin{proof}
In lemma \ref{1}, using LaSalle's invariance principle, we proved that the system (\ref{wp}) is globally asymptotically stable. This implies that there exists no trajectory within the set $\mathcal{I}$, which is unbounded, and hence, it follows that $\|x(t)\|<\infty$ which also implies from lemma \ref{k} that the sequence $\|x_k-x^*\|$ is bounded for every $k>0$. It implies thus that $\|z_k\| = \|x_k-x_{k-1}\|$ is bounded for every $k>0$. From lemma \ref{1}, we also know that 
\begin{equation}
\lim_{t\rightarrow\infty}\|\dot{x}(t)\| = 0.\label{lip}
\end{equation} 
Applying the symplectic discretisation of $\dot{x}(t) = \frac{x_{k+1}-x_{k}}{\sqrt{s}}$, we obtain for a sufficiently small $s>0$,
\begin{equation}
 \lim_{k\rightarrow\infty} \|z_k\| = 0,
    \label{e2}
\end{equation}
as required.
\end{proof}
\begin{lemma}\label{jp}
 The inner product sequence: $\langle x_{k}-x^*,z_k\rangle$ converges as $k\rightarrow \infty$.
\end{lemma}
\begin{proof}
Let the sequence $x_k-x^* = p_k$. Then, using Cauchy-Schwarz inequality for finite dimensional vector spaces, we have
\begin{equation*}
   \langle p_k,z_k \rangle = \sum_{i=0}^{n} p^i_k z^i_k \leq \left(\sum_{i=0}^{n}p_k^i\right)^{\frac{1}{2}} \left(\sum_{i=0}^{n}z_k^i\right)^{\frac{1}{2}}=\|p_k\|\|z_k\|.
\end{equation*}
Now applying the limits using the Cauchy-Schwarz inequality, we have
\begin{equation*}
    -\lim_{k\rightarrow \infty} \|p_k\| \lim_{k\rightarrow \infty}\|z_k\|\leq\; \lim_{k\rightarrow \infty} \langle p_k,z_k\rangle\;\leq \lim_{k\rightarrow \infty} \|p_k\| \lim_{k\rightarrow \infty}\|z_k\|,
\end{equation*}
which implies using lemmas \ref{k} and \ref{m},
\begin{equation}
     \lim_{k\rightarrow \infty} \langle p_k,z_k\rangle = 0,
\end{equation}
as required.
\end{proof}
\begin{theorem}\label{t1}
The momentum sequence of the whiplash gradient descent $\|z_k\|$ converges at the rate $\mathcal{O}\left(\frac{1}{\sqrt[3]{k}}\right)$ for an $L$-smooth convex cost function.
\end{theorem}
\begin{proof}
From (\ref{WGDA}), using $s= \frac{1}{L}$, we obtain
\begin{gather*}  
    k\|z_k\|^2 z_k = -\left(L z_{k+1}-(L-\sqrt{L})z_k\right)-\nabla{f}(x_k).
\end{gather*}  
Taking the norm on both sides and applying the triangle inequality, we obtain
\begin{equation*}
    k\|z_k\|^3 \leq L\|z_{k+1}\|+ L' \|z_k\|+\|\nabla{f}(x_k)\|,
\end{equation*}
where $L' = (L-\sqrt{L})$. Now, we take limits on both sides of the inequality to obtain
\begin{equation}\label{ini}
    \lim_{k\rightarrow\infty} k\|z_k\|^3\leq\lim_{k\rightarrow\infty}\left(L\|z_{k+1}\|+ L' \|z_k\|+\|\nabla{f}(x_k)\|\right).
\end{equation}
Using (\ref{e2}), we simplify using the property of subadditivity, the above inequality to
\begin{equation}\label{om}
    0\leq\lim_{k\rightarrow\infty} k\|z_k\|^3\leq\lim_{k\rightarrow\infty}\|\nabla{f}(x_k)\|.
\end{equation}
Now, using (\ref{e1}) for a Lipschitz continuous function $f(x)$, we obtain
\begin{equation*}
    0\leq\lim_{k\rightarrow\infty} k\|z_k\|^3\leq\|\nabla{f}(x^*)\|.
\end{equation*}
Since $\|\nabla{f}(x^*)\| = 0$, we obtain
\begin{equation}
    \label{e3}
    \lim_{k\rightarrow\infty} k\|z_k\|^3 = 0,
\end{equation}
which implies that the sequence $\|z_k\|$ converges at the rate of $o\left(\frac{1}{\sqrt[3]{k}}\right)$ as required. 
\end{proof}
\begin{remark} \label{z}
Note that we have a strict convergence rate for the momentum sequence, which is not the case for the sequence $\|x_k-x^*\|$ itself, i.e. the system slows down as it reaches the minima, which allows us to design a stopping criterion for convex functions with respect to the momentum of the system. In algorithm (\ref{alg:1}), instead of the naive approach for iterating the algorithm we could use $\|z_k\|\leq\epsilon$ as the stopping criterion, where $\epsilon\geq0$ is a small value. An early stopping could significantly reduce computations as per the requirements of the subroutine in which the algorithm is deployed, as shown in algorithm:\ref{alg:2}.
\end{remark}
\begin{corollary}
The sequence $|\alpha_k|$ of the Whiplash descent method converges to $\left(1-\frac{1}{\sqrt{L}}\right)$ as $k\rightarrow\infty$ for any $L$-smooth convex cost functions.
\end{corollary}
\begin{proof}
It follows from theorem \ref{t1} that if the sequence $\|z_k\|$ is convergent at the rate $o\left(\frac{1}{\sqrt[3]{k}}\right)$, then the sequence $M_k = \frac{k}{L}\|z_k\|^2\geq 0$ must be convergent for a sufficiently large $k$ where $k>L$. For $\tilde{L} = \left(1-\frac{1}{\sqrt{L}}\right)$, it follows that the sequence 
\begin{equation*}
    \lim_{k\rightarrow\infty}|\alpha_k| = |\tilde{L}-\lim_{k\rightarrow\infty} M_k| = |\tilde{L}|,
\end{equation*}
as required.
\end{proof}
\subsection{Relaxation sequences}\label{relax}
  \begin{figure}
    \centering
    \includegraphics[width = 0.65\columnwidth]{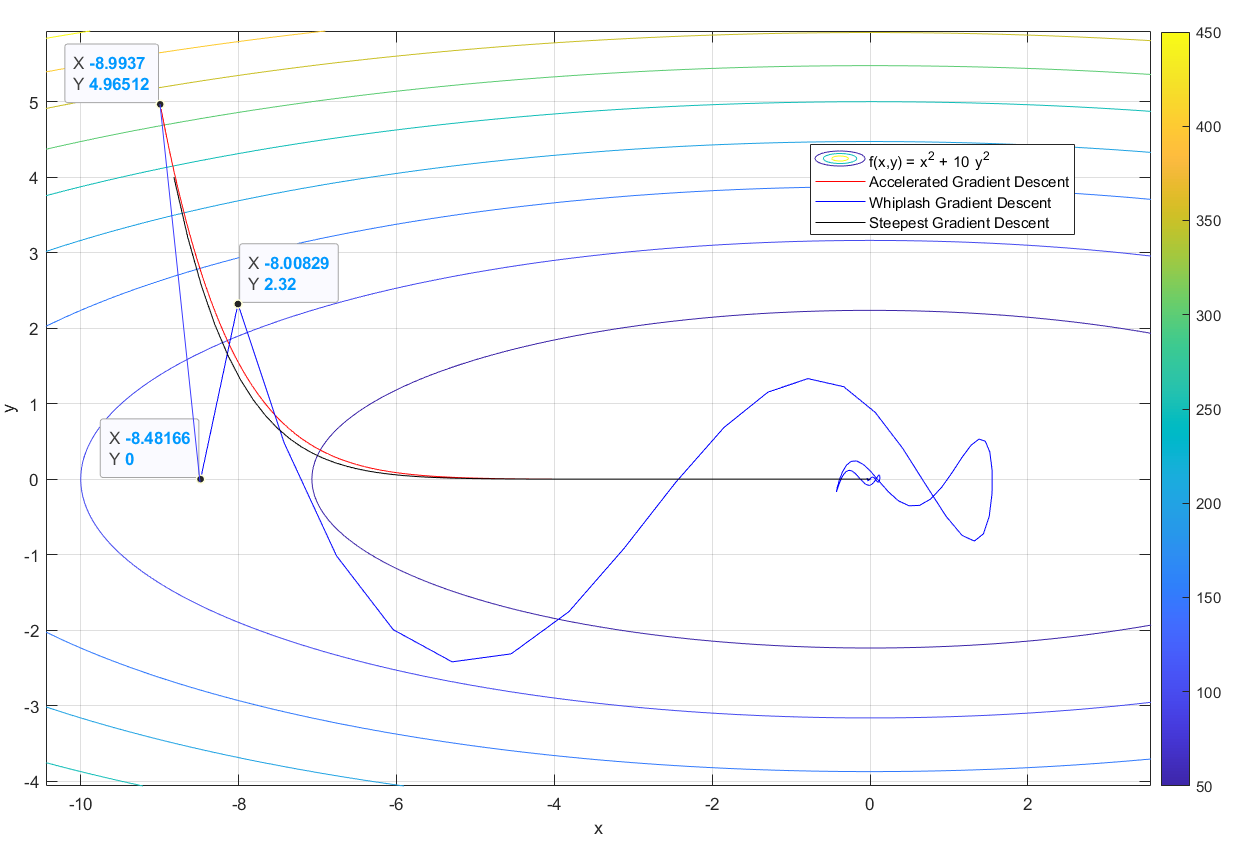}
    \caption{The whiplash scheme for ill-conditioned geometries as a comparative study. The optimiser finds the valley within a single iteration for step-size $s=0.01$.}
    \label{fig: comp}
\end{figure}
Using numerical experiments, we observe that, unlike other accelerated methods, the first step taken by the whiplash method is the largest. It increases momentum to escape the valley, after which it converges rapidly to the optima, with the oscillations dampened. The relaxation of the convergent sequence causes such a phenomenon. Closed-loop methods evolve with the momentum, which assists in escaping low curvature geometries by increasing the amplitude of transient oscillations. For higher curvature, this leads to rapid convergence towards the nearest minima---these characteristics set apart the closed-loop inertial gradient descent algorithms from classical methods. A comparative study, as shown in Figure: \ref{fig: comp}, illustrates this effect. The following theorems introduce \textit{relaxation sequence}s for $L$-smooth cost functions, a set of discrete energy terms which are strictly positive, bounded and convergent and allow us to study momentum-based scaling for the whiplash method. First we shall start with a quadratic cost function, to showcase what we mean by relaxation. 
\begin{theorem}\label{t2}
The norm of the momentum sequence $\|z_k\|$ of the Whiplash algorithm is upper and lower bounded by the recurring sequences for a cost function $f(x) =\frac{1}{2}\lambda \|x\|^2$, where $\lambda\in\mathbb{R}^{+}_{*}$ as
\begin{equation}\label{gop}
    \left||\alpha_k|\|z_k\|-\underbrace{\lambda \|x_k\|}_{I}\right|\leq \|z_{k+1}\| \leq |\alpha_k|\|z_k\| + \underbrace{\lambda\|x_k\|+\sqrt{2|\alpha_k| \lambda \|x_k\|\|z_k\|}}_{II}\quad\forall\,k>0,
\end{equation}
for any choice of step size $0<s \leq 1$.
\end{theorem}
\begin{proof}
    From (\ref{WGDA}) for the given choice of gradient for the cost function in question, we have, $\nabla f(x) = \lambda x$, we have upon rearrangement and consolidation of the terms,
    \begin{equation*}
        x_{k+1} = \left(1-\lambda + \alpha_k \right) x_k - \alpha_k x_{k-1},
    \end{equation*}
    which upon a simple manipulation yields
    \begin{equation}\label{pipo}
        z_{k+1} = \alpha_k z_k-\lambda x_k.
    \end{equation}
    We now take the norm on both sides and use the triangle inequality to obtain the lower bound for all $k>0$ as
    \begin{equation}\label{lhs}
         \|z_{k+1}\| \geq \left||\alpha_k| \|z_k\|-\lambda \|x_k\|\right|,
    \end{equation}
    as required. Next, we take the norm and square both sides in (\ref{pipo}), such that upon expansion of the terms we have
    \begin{equation*}
        \|z_{k+1}\|\leq \alpha^2_k\|z_k\|^2+\lambda^2\|x_k\|^2+2\langle \lambda x_k, \alpha_k z_k\rangle.
    \end{equation*}
    Using the Cauchy-Schwarz inequality, and the inequality
    $a^2\leq b^2+c^2\implies a\leq b+c\;\forall\,a,b,c\in\mathbb{R}^{+}$, we have
    \begin{equation}\label{rhs}
        \|z_{k+1}\|\leq |\alpha_k|\|z_k\|+\lambda\|x_k\|+\sqrt{2|\alpha_k|\lambda\|x_k\|\|z_k\|}\;\forall\,k>0.
    \end{equation}
    Upon combining (\ref{lhs}) and (\ref{rhs}), we have (\ref{gop}), as required.
\end{proof}
Note that the terms $(I)$ and $(II)$ in (\ref{gop}) are what signify the \emph{relaxation scheme}, making the Whiplash descent algorithm essentially non-classical. There terms occur because of the feedback and as we shall show in theorem \ref{t3}, we shall find similarly occurring sums and terms which relaxes the convergence.

\begin{remark} 
The physical significance of this result is that the momentum sequence is not monotonically decreasing but as we find from the earlier lemmas, the system is convergent, which means that there is a relaxation between the bounds for the system, essentially allowing it to adapt its behaviour with respect to the geometry.
\end{remark}
\begin{theorem}\label{t3}
The whiplash gradient scheme admits a absolutely convergent relaxation sequence $\Delta_k$ for convex $L$-smooth cost functions, so that
\begin{gather*} 
    \|x_{k+1}-x^*\|^2 \leq \|x_k-x^*\|^2 + \Delta_k,\quad \forall\;k>0,\quad\text{where}\;\Delta_k = 3\alpha_k^2\|z_k\|^2+
    2\alpha_k\|x_{k+1}-x^*\|\|z_k\|.
\end{gather*} 
\end{theorem}
\begin{proof}
From (\ref{WGDA}) and (\ref{ls}), we know that
\begin{equation*}
\begin{gathered}
x_{k+1} = x_k -\frac{1}{L}\nabla{f}(x_k)+\alpha_k z_k.
\end{gathered}
\end{equation*}
Subtracting the optimiser $x^*$ on both sides, taking its norm and squaring, we obtain
\begin{gather*}
\|x_{k+1}-x^*\|^2 = \|x_k-x^*+\alpha_k z_k\|^2-\frac{2}{L}\langle x_k-x^*, \nabla{f}(x_k) \rangle-\frac{2\alpha_k}{L} \langle \nabla{f}(x_k), x_k-x_{k-1}\rangle+\frac{1}{L^2}\| \nabla{f}(x_k)\|^2.
\end{gather*}  
Using (\ref{bomb}) and dropping the strictly negative terms,  we obtain
\begin{gather*}
    \|x_{k+1}-x^*\|^2 \leq \|x_k-x^*+\alpha_k z_k\|^2-
    \frac{2\alpha_k}{L} \langle \nabla{f}(x_k), z_k\rangle.
\end{gather*}
Rearranging the terms, we obtain
\begin{equation}
\begin{gathered}
    \|x_{k+1}-x^*\|^2 \leq \|x_k-x^*\|^2+\alpha_k^2\|z_k\|^2+
    2\alpha_k \langle \left(x_k-x^*-\frac{\nabla{f}(x_k)}{L}\right), z_k \rangle.
\end{gathered}
\end{equation}
On further rearrangement using the algorithm (\ref{WGDA}), we have
\begin{equation}
    \|x_{k+1}-x^*\|^2 \leq \|x_k-x^*\|^2+ 3\alpha_k^2\|z_k\|^2+
    2\alpha_k\langle x_{k+1}-x^*,z_k\rangle.
\end{equation}
Following the approach in lemma \ref{jp}, we have
\begin{equation}
    \|x_{k+1}-x^*\|^2 \leq \|x_k-x^*\|^2+
    \underbrace{3\alpha_k^2\|z_k\|^2+
    2\alpha_k\|x_{k+1}-x^*\|\|z_k\|}_{\Delta_k}.
\end{equation}
Now using this relationship recursively, we have the telescopic sum, which upon further simplification gives
\begin{equation}
\label{bound}
\begin{gathered}
\|x_{k+1}-x^*\|^2 \leq \|x_0-x^*\|^2+\sum_{i=0}^{k+1}\Delta_i.
\end{gathered}
\end{equation}
Now, let us investigate the distance sequence $\Delta_k$. Using the algorithm (\ref{WGDA}), we may substitute and rearrange the terms as the following:
\begin{equation*}
\Delta_k = 3\alpha_k^2\|z_k\|^2+
    2\alpha_k\|x_{k+1}-x^*\|\|z_k\|.
\end{equation*}
Taking the limit on both sides, we have:
\begin{equation*}
    \begin{gathered} \lim_{k\rightarrow\infty}
    \Delta_k = 3\lim_{k\rightarrow\infty}\alpha_k^2\|z_k\|^2+2\lim_{k\rightarrow\infty} 2\alpha_k\|x_{k+1}-x^*\|\|z_k\|.
\end{gathered}
\end{equation*}
Now using lemmas \ref{k}, \ref{m}, \ref{jp} and theorem \ref{t1}, we have that the \emph{relaxation sequence} (distance) $0\leq\Delta_k<\infty$ for every $k>0$ and is convergent, as required. It follows immediately that
$$
\lim_{N\rightarrow \infty}\sum_{k=1}^{N}|\Delta_k| = \sum_{k=1}^{\infty}\Delta_k<\infty,
$$
which implies that the sequence $\Delta_k$ is absolutely convergent as required.
\end{proof}
\begin{remark}
This shows that we do not have a clear nature to the type of convergence (rate) that evolves over the iterations for convex objectives. A similar result is observed in the next theorem, where we use a strong assumption of local convexity.
\end{remark}
\begin{theorem}\label{t4}
Let \(f\) be an $L$-smooth function, which satisfies the Polyak-{\L}ojasiewicz inequality (\ref{PL}) for some $L>\lambda>0$. Then the whiplash gradient descent method converges at a rate, given by the relaxation sequence $\xi_i>0$, for every $i>0$:
\begin{equation}
|f(x_{k+1})-f^*|  \leq  \left(1-\frac{\lambda}{L}\right)^k |f(x_0)-f^*|+\frac{L}{2}\sum_{i = 1}^k \left|\xi_i\right|,
\end{equation}
for every $k>0$ where 
$$\xi_{i} = \left( 1-\frac{\lambda}{L}\right)^{i} \left(1-\frac{1}{\sqrt{L}} -\frac{{k-i}}{L}\|z_{k-i}\|^2\right)^2\|z_{k-i}\|^2,$$
is an absolutely convergent sequence.
\end{theorem}
\begin{proof}
Using (\ref{WGDA}) and (\ref{ls}), we have 
\begin{equation*}
    f(x_{k+1})\leq f(x_k)+\langle \nabla{f(x_{k})},z_{k+1} \rangle+\frac{L}{2}\|z_{k+1}\|^2.
\end{equation*}
Upon expanding $z_{k+1}$ and rearranging, we obtain
\begin{gather*}
    f(x_{k+1})\leq f(x_{k})+\langle \nabla{f(x_{k})},\alpha_k\,z_k-\frac{1}{L}\nabla{f(x_k)} \rangle + \frac{\alpha_k^2 L}{2}\|z_k\|^2+\frac{1}{2L}\|\nabla{f}(x_k)\|^2-\langle \alpha_k z_k,\nabla{f}(x_k)\rangle.
\end{gather*}
Upon cancellation of common terms and subtracting from both sides the optimal value $f^*$ and rearrangement, we obtain
\begin{equation*}
    f(x_{k+1})-f^*\leq f(x_{k})-f^*-\frac{1}{2L}\|\nabla f(x_k)\|^2+\frac{\alpha_k^2 L}{2}\|z_k\|^2.
\end{equation*}
Now, applying the P{\L} inequality (\ref{PL}), we obtain
\begin{equation*}
f(x_{k+1})-f^*\leq f(x_{k})-f^*-\frac{\lambda}{L}\left(f(x_k)-f^*\right)+\frac{\alpha_k^2 L}{2}\|z_k\|^2.
\end{equation*}
Upon further rearrangement, we obtain
\begin{equation*}
    f(x_{k+1})-f^*
    \leq \left(1-\frac{\lambda}{L}\right)(f(x_k)-f^*)+\frac{\alpha_k^2 L}{2}\|z_k\|^2.
    \end{equation*}
Applying this recursively, we obtain the telescoping sum,
\begin{equation*}
f(x_{k+1})-f^*  \leq  \left(1-\frac{\lambda}{L}\right)^k(f(x_0)-f^*)+\frac{L}{2}\sum_{i = 1}^k \underbrace{\left( 1-\frac{\lambda}{L}\right)^{i} \left(1-\frac{1}{\sqrt{L}} -\frac{{k-i}}{L}\|z_{k-i}\|^2\right)^2\|z_{k-i}\|^2}_{\xi_i},
\end{equation*}  
for all $k>0$ as required. Taking the absolute values on both sides, and applying the triangle inequality, we have
\begin{equation}
    |f(x_{k+1})-f^*|  \leq  \left(1-\frac{\lambda}{L}\right)^k|f(x_0)-f^*|+\frac{L}{2}\left|\sum_{i = 1}^k \xi_i\right|.
\end{equation}
Now, we know that every term in $\xi_i$ is positive from the relation of $L>\lambda$ and using lemmas \ref{k} and \ref{m}, we know that every term in the relaxation sequence $\left|\xi_i\right|$ must converge for every $i>0$ and thus $|f(x_{k+1})-f^*|$ must converge at the given relaxed rate. Moreover, from lemma \ref{k}, we know that every term in $\xi_i$ for every $i>0$ is such that
\begin{equation*}
    \lim_{k\rightarrow\infty}\max_{i\leq k}|\xi_{i}| <\infty,
\end{equation*}
which implies that 
\begin{equation}
\lim_{N\rightarrow \infty}\left|\sum_{i = 1}^N \xi_i\right| = \sum_{i = 1}^{\infty} \left|\xi_i\right|<\infty.
\end{equation}
Hence, the sequence $\xi_i$ is absolutely convergent, as required.
\end{proof}
\begin{remark}
For a specific class of functions which is strictly convex, the algorithm exhibits an implicit nature for the convergence upper bounds. 
\end{remark}
\begin{theorem}\label{t5}
For the convex class of cost functions, $g(x) = \frac{1}{4}\lambda \|x-x^*\|^4$, where $\lambda \in \mathbb{R}^+$ is such that for sufficiently large $k>\lambda$, the sequence $\|x_k-x^*\|$ of the Whiplash Gradient descent converges at the rate of
\begin{equation*}
   \mathcal{O}\left(\frac{1}{\sqrt[3]{1-\frac{\lambda}{k}}}\right).
\end{equation*}
\end{theorem}
\begin{proof}
For the function $g(x)$, we have the gradient as $\nabla{g}(x) = \lambda \|x-x^*\|^3$. From the inequality (\ref{ini}), using the sequence $p_k = x_k-x^*$, and taking the absolute value on both sides, we have
\begin{equation*}
     k\|z_k\|^3\leq \left|L\|z_{k+1}\|+L'\|z_k\|+\lambda\|p_k\|^3\right|.
\end{equation*}
Using the relation of $z_k = p_k-p_{k-1}$ and the triangle inequality, we have
\begin{equation*}
     k\|p_k-p_{k-1}\|^3\leq \left|L\|z_{k+1}\|+L'\|z_k\|\right|+\lambda\|p_k\|^3.
\end{equation*}
We know that for two vectors $a, b$ we have $$\|a\|^3-\|b\|^3 \leq \left|\|a\|^3-\|b\|^3 \right| \leq \|a-b\|^3.$$
Using this relationship, we have
\begin{equation*}
    k\left(\|p_k\|^3-\|p_{k-1}\|^3\right) \leq \left|L\|z_{k+1}\|+L'\|z_k\|\right|+\lambda\|p_k\|^3.
\end{equation*}
Upon rearrangement, we have
\begin{equation*}
        0\leq \left| 1-\frac{\lambda}{k} \right| \|p_k\|^3 \leq \left|\frac{L\|z_{k+1}\|+L'\|z_k\|}{k}\right|+\|p_{k-1}\|^3\;\forall\;k>\lambda.
\end{equation*}
From lemma \ref{2}, we know that the sequence $\|p_k\|$ is convergent and from theorem \ref{t1}, we have the sequence convergence rate for $\|z_k\| = \mathcal{O}\left(\frac{1}{\sqrt[3]{k}}\right)$. We know using the triangle inequality that
\begin{equation}
   \lim_{k\rightarrow\infty} \left|\frac{L\|z_{k+1}\|+L'\|z_k\|}{k}\right|\leq \lim_{k\rightarrow\infty}\left|\frac{L\|z_{k+1}\|}{k}\right|+\lim_{k\rightarrow\infty}\left|\frac{L'\|z_k\|}{k}\right| = 0.
\end{equation}
Upon applying limits on both sides, we thus have
\begin{equation*}
        0\leq\lim_{k\rightarrow\infty} \left(1-\frac{\lambda}{k}\right)\|p_k\|^3 \leq \lim_{k\rightarrow\infty} \|p_{k-1}\|^3=0\;\forall\;k>\lambda.
\end{equation*}
Hence, we conclude that for all $k$ sufficiently larger than $\lambda$, we have
\begin{equation}
    \|p_k\| = \|x_k-x^*\| = \mathcal{O}\left(\frac{1}{\sqrt[3]{1-\frac{\lambda}{k}}}\right),
\end{equation}
as required.
\end{proof}
\section{Envelope convergence}\label{sec5}
In the study of inertial gradient dynamics, we know that such systems are globally asymptotically stable for all strictly positive damping laws. We further know that if the cost function gradient is Lipschitz continuous, there exists a convergent global solution, for such a system. But such a solution is non-analytical in many cases. Our analysis starts with consideration of the asymptotic nature of globally asymptotically stable solutions of (\ref{gen}). To understand the rate of convergence of the system, we need to perform a Lyapunov analysis on the Lipschitz continuous scaled system 
\begin{equation}\label{system}
    y(t) = \mathcal{P}(t) x(t).
\end{equation}
As discussed earlier, finding such Lyapunov functions is tedious and often lacks motivation. Lyapunov's fundamental argument leads to a family of energy functions that are non-increasing along the system's dynamics, typically the only constraint. Therefore, instead of finding a state-dependent energy function for the scaled dynamical system, we introduce an integral anchor constraint. Further, we predict convergence rates using the knowledge of \emph{implicit time dependence} and the \emph{system states' bounded nature}.
\subsection{Asymptotic loose bound}\label{alb}
We will consider the asymptotic behaviour of the system and, correspondingly, the associated Lyapunov functions. It is, therefore, necessary to introduce a concept of asymptotic loose bound of functions in this context. We have modified it reasonably to suit our requirement as $\mathscr{O}$ while borrowing some of its characters. A function $f(t)$ which is asymptotically loose bounded, is said to be convergent at the rate $\nu(t)>0$, if 
\begin{equation}
0\leq\varlimsup_{t\rightarrow\infty} |\nu(t)f(t)|< \infty\;\text{which is denoted as}\; f(t)=\mathscr{O}\left(\frac{1}{\nu(t)}\right).
\label{def}
\end{equation} 
We define the loose bound notation of two real and bounded functions $\mathcal{A}_t, \mathcal{B}_t:\mathbb{R}^{+}\mapsto \mathbb{R}$, such that $\mathcal{A}_t = \mathcal{O}\left(a(t)\right)$ and $\mathcal{B}_t = \mathcal{O}\left(b(t)\right)$, where $\mathscr{O}\left(\mathrm{a}(t)\right)\leq\mathscr{O}\left(\mathrm{b}(t)\right)$. We further have the following relations:
\begin{equation}
\begin{gathered}\label{omega}
    \mathscr{O}\left(\lambda\,\mathrm{a}(t)\right) = \sgn(\lambda)\,\mathscr{O}\left(\mathrm{a(t)}\right) \;\text{where}\; \lambda \in \mathbb{R},\\
    \mathscr{O}\left(a(t)\right)\leq  \lambda \mathscr{O}\left(b(t)\right)\Rightarrow \mathcal{A}_t\sim \sgn(\lambda)\,\mathscr{O}\left(b(t)\right),\;\text{(asymptotically equivalent)}\\
    \mathscr{O}\left(\mathrm{a}(t)\mathrm{b}(t)\right) = \mathscr{O}\left(\mathrm{a}(t)\right) \mathscr{O}\left(\mathrm{b}(t)\right),\\
    \mathscr{O}\left(\mathrm{a}(t) \pm \mathrm{b}(t)\right) = \pm\, \mathscr{O}\left(\mathrm{b}(t)\right).
\end{gathered}
\end{equation}
This implies that, whichever function is relatively asymptotically weaker to converge is the dominating function amongst the two and shall be represented by the \say{$\asympl$} symbol i.e. $\mathcal{A}_t\asympl \mathcal{B}_t \implies\;\mathcal{B}_t= c \mathcal{A}_t$,  but $c' \mathcal{B}_t\nasympl \mathcal{A}_t$, where $\{c,c'\}\geq 0$. This relation can be extended to higher dimensions for vectors as shown in lemma \ref{vec}. 
\begin{lemma}\label{fund}
Suppose an $L$-smooth function $f(t)$ is asymptotically bounded, and convergent at the rate $\beta(t)$ such that $\varlimsup_{t\rightarrow\infty} |\beta(t) f(t)| = 0$. If the rate $\beta(t)$ is such that
\begin{equation}
    \beta(t)<o\left(\frac{1}{t}\right) \;\forall \,t>0,
\end{equation}
then there exists $c\geq0$ such that 
\begin{equation}
    \int_{t_0}^t |f(s)| ds <c, \quad \forall\;t>t_0\geq 0.
\end{equation}
\end{lemma}
 
\begin{proof}
Since $|f(t)|$ is everywhere continuous and defined for all $t>0$, it follows that for any $ m >t_0 \geq 0$ and any $m \leq t$, we obtain
\begin{equation*}
 \int_{t_0}^t |f(s)| ds = \int_{t_0}^m |f(s)| ds + \int_{m}^t |f(s)| ds.
\end{equation*}
Now, since $|f(t)|$ is everywhere bounded, it follows that there exists a $c_1>0$, such that $\int_{t_0}^m |f(s)| ds = c_1$ and there exists $c_2>0$ such that for a sufficiently large $m$, we obtain using (\ref{def}) and (\ref{omega}) that
\begin{equation*}
    \int_{m}^t |f(s)| ds \asympl \int_{m}^t |\mathscr{O}\left(\frac{1}{\beta(s)}\right)| ds \leq c_2 \quad \forall\;m\leq t.
\end{equation*}
The result then follows with $c_1 + c_2 = c$.
\end{proof}
\begin{lemma}\label{bi}
In addition to the conditions on $f$ in lemma \ref{fund}, if we have the conditions $\lim_{t\rightarrow\infty}|\dot{f}(t)| = 0$, then $ \dot{f}(t) \asympl -\mathscr{O}\left(\frac{\dot{\nu}}{\nu^2}\right)$.
\end{lemma}
\begin{proof}
Consider the function $\mu(t) = \nu(t)f(t)$. From lemma \ref{fund}, we know that $\mu(t)$ is bounded and must converge as $t\rightarrow\infty$. Upon differentiation of $\mu(t)$, we obtain $\dot{\mu}(t)= \dot{\nu}f+\dot{f}\nu$. We know that $\mu (t)$ is Lipschitz continuous. This implies that
\begin{equation*}
\varlimsup_{t\rightarrow\infty}\dot{\mu}(t) = 0.
\end{equation*}
It follows upon rearrangement that
\begin{equation*}
    \varlimsup_{t\rightarrow\infty}\dot{f} = \varlimsup_{t\rightarrow\infty} -f\frac{\dot{\nu}}{\nu}.
    \end{equation*}
This implies using (\ref{def}) that
    \begin{equation}
\dot{f} \asympl -\mathscr{O}\left(\frac{\dot{\nu}}{\nu^2}\right),
\end{equation}
as required.
\end{proof}
\begin{lemma}\label{vec}
Let us consider a finite-dimensional vector-valued $L$-smooth function $ q(t):\mathbb{R}^{+}\rightarrow\mathbb{R}^d$, which is asymptotically bounded, and converges to the vector $q^*$ such that  $\|\mathcal{Q}(t)\| = \mathscr{O}\left(\frac{1}{\nu(t)}\right)$, where $\mathcal{Q}(t) = q(t)-q^*$, it follows that $\mathcal{Q} \asympl \mathscr{O}\left(\frac{1}{\nu(t)}\right)$.
\end{lemma}
\begin{proof}
As $\mathcal{Q}(t)$ is a finite-dimensional vector-valued function, there exist scalar functions $f_i: \mathbb{R}^{+}\mapsto \mathbb{R}$ which are everywhere bounded, $\lim_{t\rightarrow\infty} f_i(t) = 0$ for every $i\geq 1$ and for which
\begin{equation}\label{jem}
    \mathcal{Q}(t) = \sum_{i = 1}^{d}\,f_i(t)\;\hat{e}_i,
\end{equation}
where $|\hat{e}_i| = 1$ denote the unit basis vectors. We also know that 
\begin{equation}
    \|\mathcal{Q}(t)\| = \sqrt{\sum_{i = 1}^{d}f_i^2(t)}.
    \label{norm}
\end{equation}
Using Cauchy-Schwarz inequality on (\ref{jem}), and (\ref{norm}), it follows that
\begin{equation}
       \mathcal{Q}(t) \leq \sqrt{\sum_{i = 1}^{d}\,f_i^2(t)}\sqrt{\sum_{i = 1}^{d}\,\hat{e}_i^2} = \sqrt{d} \|\mathcal{Q}\| = \sqrt{d}\mathscr{O}\left(\frac{1}{\nu(t)}\right),
        \label{rate}
\end{equation}
which implies that as $t$ grows, for a finite $d\geq1$, we have
\begin{equation}
\mathcal{Q}(t)\leq \sqrt{d}\mathscr{O}\left(\frac{1}{\nu(t)}\right).
\end{equation}
Therefore, it follows using (\ref{omega}) that $\mathcal{Q}(t) \asympl \mathscr{O}\left(\frac{1}{\nu(t)}\right)$, as required.
\end{proof}
\begin{lemma}\label{grab}
For a convex cost function with Lipschitz continuous gradients, if $\|x_t-x^*\|$ converges at a rate of $\frac{1}{\lambda_t}$, then the gradient of the function converges at least at that rate.
\end{lemma}
\begin{proof}
From (\ref{2}) we know that,
\begin{equation}
     \|\nabla f(x_t)-\nabla{f}(x^*)\|\leq L\|x_t-x^*\|,
\end{equation}
If the term $x-x^*$ converges at the rate $\frac{1}{\lambda_t}$, where $\lambda_t$ represents the decay rate (increasing function), then for a convex function with $x^*$ as the optima, there must exist a constant $C$, such that
        \begin{equation}
        \begin{gathered}
        \|\nabla{f}(x_t)\| \leq
        \frac{C}{\lambda_t}\quad\text{where}\;C\in\mathbb{R^+}\\
        \implies \|\nabla{f}(x)\| = \mathcal{O}\left(\frac{1}{\lambda_t}\right),
    \end{gathered}
\end{equation}
as required.
\end{proof}
\begin{lemma}\label{9}
For a convex cost function with Lipschitz continuous gradients, if $\|x_t-x^*\|$ converges at a rate of $\frac{1}{\mu_t}$, then $\dot{\mu}_t\langle x_t-x^*, \dot{x}_t\rangle$ is also convergent. 
\end{lemma}
\begin{proof}
From Lyapunov's Second theorem we know that if for a system's solution $x(t)$, which is globally asymptotically stable (lemma \ref{1}) and converges at a particular rate to $x^*$, then the derivative is also stable and convergent at that rate \cite{Kalman1960}. This means that if $\|x_t-x^*\|$ converges at a rate of $\frac{1}{\mu_t}$, then $\|\dot{x}(t)\|$ converges at that rate as well. We have by Cauchy-Schwarz theorem, that as $t>0$ grows larger,
\begin{equation}
        \langle x_t-x^*, \dot{x}\rangle \leq \|x_t-x^*\|\|\dot{x}_t\| = \frac{C}{\mu^2_t},\;C\in\mathbb{R}^{+}.
    \end{equation}
Now, we know from mean value theorem that for a strictly positive and increasing function $\mu_t$, the function $g_t = \mu^2_t$ increases arbitrarily as
    \begin{equation}
        \dot{g}_t = 2\mu_t\dot{\mu}_t,
    \end{equation}
which implies that it increases arbitrarily faster by $\mu_t$ by a factor of $2\mu_t$. This implies that for large values of $t>0$, $$\mu^2_t>\dot{\mu}_t.$$ It follows that
\begin{equation}
    \varlimsup_{t\rightarrow\infty}\dot{\mu}_t\langle x_t-x^*, \dot{x}_t\rangle = 0,
\end{equation}
as required.
\end{proof}
\subsection{Method of envelope convergence}
In \cite{Whiplash}, we defined the converger $\mathcal{P}(t)>0$, such that for a convex cost function $f$, we have a convergence rate of $\mathcal{O}\left(\frac{1}{\mathcal{P}(t)}\right)$. This converger must be defined $\forall\;t>0$, such that $\mathcal{P}(t)\rightarrow 0$ as $t\rightarrow \infty$. A simple way of numerically verifying these results is by \textit{enveloping} the output. We show a convergence rate for $x(t)$; i.e. we define $y(t) = \mathcal{P}(t)(x(t)-x^*)$, and demonstrate its numerical convergence. The intuition is that if $y(t)$ is convergent, then $x(t)-x^*$ must converge at least at the rate of $\mathcal{O}\left(\frac{1}{\mathcal{P}(t)}\right)$.

In this method, we are extending the concept of using a solution to investigate a convergence rate for its inertial gradient system, as presented in \cite{attouch2021fast}. We consider an inertial gradient dynamical system of the form:
\begin{equation}
    \ddot{x}(t)+(1+\alpha(t)) \dot{x}(t)+\nabla{f}(x(t)) = 0,
    \label{gen}
\end{equation}
The scalar function $\alpha (t)$ is a closed-loop control law.
We describe the envelope convergence method as the following: 
\begin{enumerate}
\item We check for a converger $\mathcal{P}(t)$ for the system (\ref{gen}), such that $y(t) = \mathcal{P}(t)x(t)$ is numerically convergent. 
\item We consider two types of rates: exponential (linear) and polynomial (sub-linear), characterised by $\eta>0$ as the \textit{strength of convergence} in continuous-time. This means that an exponential rate $\mathcal{O}\left(e^{-\eta t}\right)$ is stronger than a polynomial rate  $\mathcal{O}\left(\frac{1}{t^{\eta}}\right)$. We start with small values of the exponential rate $\eta>0$. If that leads to instability, we consider polynomial rates. We have devised a simple strategy to predict these convergers computationally as follows:
\begin{itemize}
\item \textit{Exponential rate}: We start with $\eta = 0.1$ and keep running simulations for increasing values of $\eta$ by steps of size $0.05$. If a transition to stability exists in this region, it is sharp and moves from bounded to divergent solutions within a range of $0.01$. Hence, once the transition is observed, we vary $\eta$ by steps of size $0.001$. We check different starting conditions if stability is observed for a critical range.
\item \textit{Polynomial rate}: We start with $\eta = 0.5$ and keep running simulations for increasing values of $\eta$ by steps of size $0.1$. If a transition to stability exists in this region, it is sharp and moves from bounded oscillations to divergent within a range of $0.05$. Hence, once the transition is observed, we vary $\eta$ by steps of size $0.01$. We check different starting conditions if stability is observed for a critical range.
\end{itemize}
    \item Once we have verified this, we can define a system such that there exists an energy function $\mathcal{E}(t)$, which for a particular integral anchor constraint, proves the convergence rate of the objective value. The envelope convergence method has been visualized in Figure: \ref{fig: flow}. 
\end{enumerate} 
\begin{remark}
    It is important to note that the only short coming of this manual tuning process is progressively checking convergence for smaller step-sizes. It is possible that if the step-size is decreased, as long as the system provides an output without having any discontinuity in integration, it is possible that the system (especially if it is ill-conditioned) might converge for stronger rates of convergence.
\end{remark}
\begin{figure}
    \centering
    \includegraphics[width = 0.65\columnwidth]{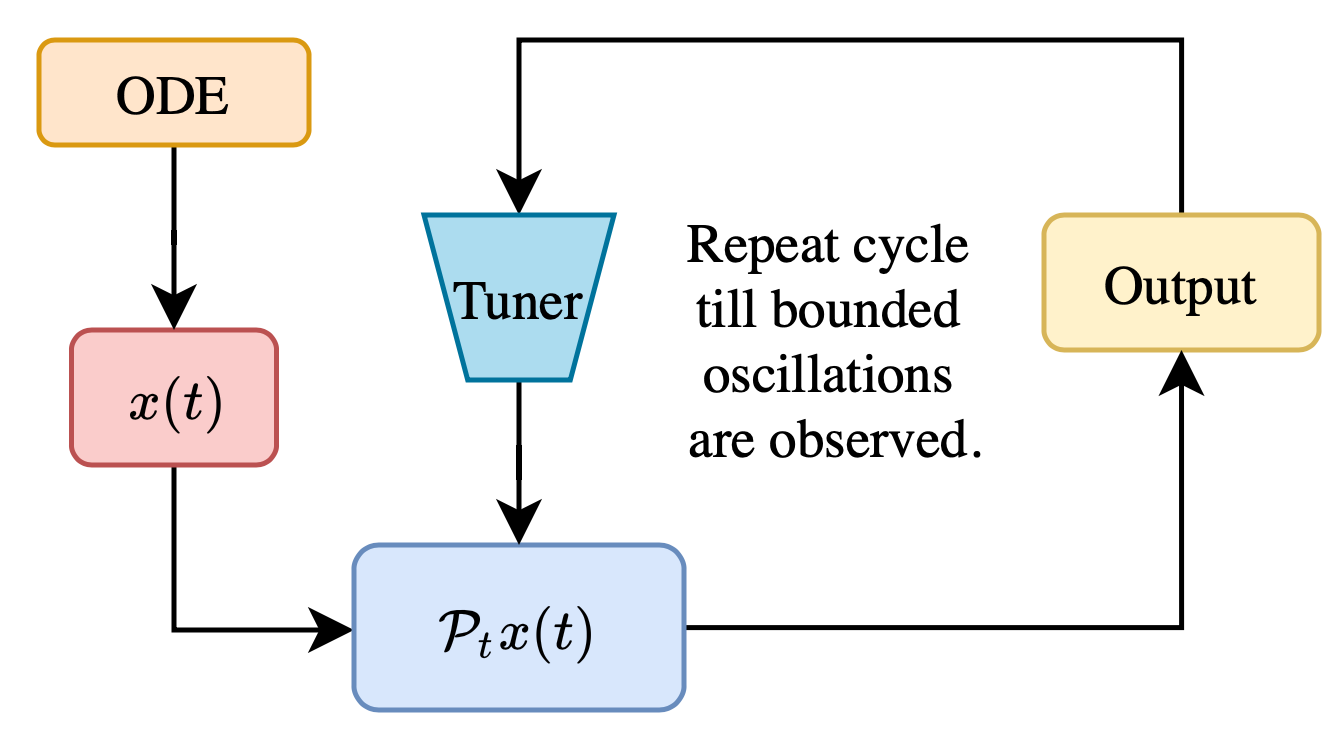}
    \caption{Method of envelope convergence where we observe transition and tune or update the strength manually to find appropriate convergers.}
    \label{fig: flow}
\end{figure}
    \begin{theorem}{(Lyapunov Rate Method)}\label{tt}
    For a twice differentiable convex cost function $f$, the system (\ref{gen}) converges with a rate of:
    \begin{equation}
        f(x)-f^* = \mathcal{O}\left(\frac{1}{\mathcal{P}(t)}\right),
    \end{equation} 
    for a converger $\mathcal{P}(t)>0$ and $\dot{\mathcal{P}}(t)>0$, such that $\mathcal{P}(t) x(t)$ is bounded and convergent, if there exists a constant $c \geq 0$, for all $t>0$ and for every solution to (\ref{gen}), $I_t+c\geq 0$ where
    \begin{equation}
    I_t = \int_0^t \left(\langle \alpha\mathcal{P}(s) \dot{x},x-x^*\rangle -\langle \mathcal{P}(s) \nabla f(x),\dot{x}\rangle - \frac{\dot{\mathcal{P}}(s)}{2}\|x-x^*+\dot{x}\|^2 \right)ds.
    \label{int}
    \end{equation}  
    \end{theorem}
    \vspace{-2em}
    \begin{proof}
    From (\ref{jen}), we know that
    \begin{equation*}
        f(x)-f(x^*)\leq \langle \nabla f(x), x-x^* \rangle.
    \end{equation*}
    Now, for the system (\ref{wp}), we define the energy function as:
    \begin{align}
        \mathcal{E}_t = \underbrace{\mathcal{P}(t)(f(x)-f^*)}_{\text{Potential energy}}+\underbrace{\frac{\mathcal{P}(t)}{2}\|x-x^{*}+\dot{x}\|^{2}}_{\text{Mixed Kinetic energy}}+\underbrace{I_t+c}_{\text{Stored energy}}\geq 0.
    \end{align}
Taking its time derivative, we obtain
\begin{gather*}
        \dot{\mathcal{E}}_{t}=\dot{\mathcal{P}}(t)(f(x)-f^*)+\mathcal{P}(t) \langle\nabla f(x), \dot{x}\rangle+\frac{\dot{\mathcal{P}}(t)}{2}\|x-x^*+\dot{x}\|^2
        +\mathcal{P}(t)\langle x- x^* + \dot{x},\dot{x}+\ddot{x}\rangle+\dot{I_t}.
\end{gather*}
Using (\ref{gen}) and rearranging terms, we obtain
    \begin{gather*}
        \dot{\mathcal{E}}_{t}=\dot{\mathcal{P}}(t)(f(x)-f^*-\left\langle\nabla f(x), x-x^{*}\right\rangle)+ \frac{\dot{\mathcal{P}}(t)}{2}\|x-x^*+\dot{x}\|^2-\langle \mathcal{P}(t)\alpha \dot{x}-\dot{\mathcal{P}}(t)\nabla f(x),x-x^*\rangle+\dot{I_t}.
        \end{gather*}
          
        Now, by definition
        \begin{equation}
            \dot{I_t} = \langle \alpha\mathcal{P}(t) \dot{x},x-x^*\rangle - \langle \mathcal{P}(t) \nabla f(x),\dot{x}\rangle - \frac{\dot{\mathcal{P}}(t)}{2}\|x-x^*+\dot{x}\|^2 \quad \forall\,t>0.
        \end{equation}
         Therefore, 
        \begin{equation}
            \dot{\mathcal{E}}_{t}=\dot{\mathcal{P}}(t)\left(f(x)-f^*-\left\langle\nabla f(x), x-x^{*}\right\rangle\right)\leq 0,
        \end{equation}
        $\forall\,t>0$, as $\dot{\mathcal{P}}(t)>0$. Hence, upon integration, we obtain
        \begin{equation*}
            f(x)-f^* \leq \frac{\mathcal{E}_0}{\mathcal{P}(t)} = \mathcal{O}\left( \frac{1}{\mathcal{P}(t)}\right),
        \end{equation*}
as required. 
This implies that if the functional $I_t$ has a loose bound, then there must exists a constant $c$ for which $I_t+c\geq 0$. It follows that if $I_t(x,\dot{x}, t)+c\geq 0$, for every solution $x(t)$ to the dynamics (\ref{gen}), then $\mathcal{E}_t\geq 0$ and $\dot{\mathcal{E}}_t\leq 0$ for all $t\geq 0$.
\end{proof} 
\begin{remark}
    In order to find such an $I_t$, we need to have additional assumptions on the asymptotic behaviour of the terms in $I_t$. Essentially, using the bounded nature of the terms in the energy function and the convexity of the function itself, we have proven using the smoothness assumptions and a form for the convergence rate of the system's solutions, the loose bound rate of convergence, as described in the next subsection.
\end{remark}
\subsection{Integral anchor constraint}
To consider a loose bound on the functional $I_t$, we use an asymptotic loose bound of the states $x(t)$ and $\dot{x}(t)$. Now, we use the asymptotic loose bound estimates on each term in the integrand to check if all the generalised manufactured solutions (form of convergence) \cite{attouch2021fast} for a particular choice of converger satisfy the integral anchor constraint. We consider, as before, for such an analysis the two cases: polynomial and exponential convergence settings. The following inferences lead to this hypothesis:
\begin{enumerate}
    \item All solutions of a Lyapunov stable system must be bounded and convergent.
    \item Since the gradient is Lipschitz continuous, there must exist (implicit if not explicit) solutions to the system (\ref{gen}).
    \item The solutions are at least $\mathcal{C}^2$ smooth.
    \item This implies that the must be continuously integrable in $t\in[0,\infty)$.
    \item It follows that the nature of solution is implicitly time-dependent, which implies that all solutions satisfy $\|x(t) - x^*\| = \mathscr{O}\left(\frac{1}{\lambda(t)}\right)$ for an arbitrary increasing function $\lambda(t)>0$ for all $t \geq 0$.
\end{enumerate}
Now, by lemma \ref{vec}, we therefore have $x(t)-x^*\asympl \mathscr{O} \left(\frac{1}{\lambda(t)}\right)$, which implies, using the convexity of the cost function, that $f(x)-f^* \asympl \mathscr{O}\left(\frac{1}{\lambda(t)}\right)$. From (\ref{int}), upon using Cauchy-Schwarz inequality, we find that the upper bound on the integral constraint is given as:
\begin{equation}
    I_t \asympl \int_0^t \left(\alpha \mathcal{P}(s)  \|\dot{x}\|\|x-x^*\| -\mathcal{P}(s) \|\nabla f(x)\|\|\dot{x}\| - \frac{\dot{\mathcal{P}}(s)}{2}\left(\|x-x^*\|^2-2\langle x-x^*, \dot{x}\rangle +\|\dot{x}\|^2\right) \right)ds.
\end{equation}
Now by the assumptions on the numerical convergence i.e. if (\ref{system}) numerically converges, using lemmas \ref{grab} and \ref{9}, we know that all the terms in (\ref{int}) are bounded and integrable by our inferences. And hence, they must be implicitly time-driven, asymptotically. Thus, if we can assume the nature of convergence, we can show that the Lyapunov rate method is satisfied. The following theorem considers a loose bound on the integrand, using necessary assumptions. We consider the two cases of convergence as forms without strictly asserting the specific rates.
\begin{theorem}
The integral $I_t$, which is defined for the dynamical system (\ref{gen}) for all $t>0$, for a $\mathcal{C}^2$ convex cost function $f(x)$, given by (\ref{int}), is asymptotically loose bounded for every possible solution to (\ref{gen}), i.e.
\begin{equation}
    I_t\geq c>0, 
\end{equation} if the following assumptions hold.
\end{theorem}
\textbf{Case I}: \textit{Polynomial convergence setting}: ($t\geq t_0>0$)
\begin{assumption}
$\mathcal{P}(t) = t^{\theta}$ where $\theta>0$.
\end{assumption}
\begin{assumption}
$\alpha (t) = \mathscr{O}\left(\frac{1}{t^{p}}\right)$ where $p\geq 0$.
\end{assumption}
\begin{assumption}
Consider the polynomial converger $\mathcal{P}(t) = t^{\eta}$. The \textbf{critical condition} is:\begin{gather}\label{h1}
x-x^* \asympl \mathscr{O}\left(\frac{1}{t^{\theta}}\right):\;\theta>\frac{\eta}{2}>0.
\end{gather}
\end{assumption}
\begin{proof}\label{leap}
Using lemma \ref{fund} and \ref{bi} and the assumptions above, we have 
\begin{equation}
     \dot{x}(t) \asympl -\mathscr{O}\left(\frac{1}{t^{\theta+1}}\right).
\end{equation}
We know that as $t\rightarrow\infty$, $x\rightarrow x^*$. Using lemma \ref{bi}, we have
\begin{equation}
    \nabla{f}(x) \asympl -\mathscr{O}\left(\frac{1}{t^{\theta+1}}\right),
\end{equation} 
since $f(x)-f^*$ is at least $\mathscr{O}\left(\frac{1}{t^{\theta}}\right)$.
Using (\ref{omega}), we have $\|x-x^*+\dot{x}\|^2 \asympl \mathscr{O}\left(\frac{1}{t^{2\theta+2}}\right)$. Using the assumptions above, we rewrite ($\ref{int}$) as: 
\begin{equation*}
     I_t \asympl -\int_{t_0}^t \mathscr{O}\left(s^{-2\theta+\eta-p-1}\right)+\mathscr{O}\left(s^{\eta-2-2\theta}\right)+\mathscr{O} \left(s^{-2\theta+\eta-3}\right) ds.
\end{equation*}  
Let $\tau_1 = 2\theta-\eta+p$, $\tau_2 = 2\theta-\eta+1$ and $\tau_3 = 2\theta-\eta+2$. Using (\ref{h1}) and Assumption 2, we have all $\tau_1, \tau_2, \tau_3 >0$. Upon integration and using (\ref{omega}), we obtain
\begin{equation}
     I_t \asympl \mathscr{O}\left(\frac{1}{t^{\rho}}\right),
\end{equation}
where $\rho = 2\theta-\eta+\min\{p,1\}$. It follows that for any $\rho$, the functional $I_t$ is asymptotically loose bounded, as required.
\end{proof}
\textbf{Case II}: \textit{Exponential convergence setting} ($t\geq t_0 \geq 0$)
\begin{assumption}
$\mathcal{P}(t) = e^{-\theta t}$ where $ \theta>0$. 
\end{assumption}
\begin{assumption}
$\alpha (t) = \mathscr{O}\left(e^{-pt}\right)$ where $p\geq 0$.
\end{assumption}
\begin{assumption} Consider the exponential converger $\mathcal{P}(t) = e^{\eta t}$ where $0<\eta < 2\theta$. The \textbf{critical condition} is:
\begin{gather}
x-x^* \asympl \mathscr{O}\left(e^{-\theta t}\right):\;\theta>\frac{\eta}{2}>0.
\label{h2}
\end{gather}
\end{assumption}
\begin{proof}
Using assumptions above and using lemmas \ref{fund} and \ref{bi}, we have 
\begin{equation}
    \dot{x}(t) \asympl -\mathscr{O}\left(e^{-\theta t}\right).
\end{equation}
This implies that $\|x-x^*+\dot{x}\|^2 \asympl \mathscr{O}\left(e^{-2\theta t}\right)$, and given that $f(x)-f^*$ is at least $\mathscr{O}\left(e^{-\theta t}\right)$, we have using lemma \ref{bi},
\begin{equation}
    \nabla{f}(x) \asympl -\mathscr{O}\left(e^{-\theta t}\right).
\end{equation}
Now, we rewrite (\ref{int}) as
 
\begin{gather*}
 I_t \asympl \int_{t_0}^t -\left(\mathscr{O}\left(e^{(-p+\eta-2\theta)s}\right)+\mathscr{O}\left(e^{(\eta-2\theta) s}\right)+\mathscr{O}\left(e^{(\eta -2 \theta)s}\right)\right)ds.
\end{gather*}
  
Upon integration, we obtain using (\ref{omega})
\begin{equation}
     I_t \asympl \mathscr{O} \left(e^{-\sigma t}\right).
\end{equation}
where $\sigma = 2\theta-\eta$. From (\ref{h2}) and Assumption 5, we know that $\sigma>0$. Hence, it follows that $I_t$ is asymptotically loose bounded, as required. Thus, we complete the proof of theorem \ref{tt} by proving the integral anchor constraint assumption, to include every possible solution to the dynamics (\ref{gen}), such that we may consider the rate of convergence for the objective $f(x)-f^*$, as required.
\end{proof}
\begin{remark} 
It is important to note that theorem \ref{leap} provides a condition to guarantee the convergence rate immaterial of the strength of convergence. The essential motivation for the envelope convergence method is that \emph{if a system converges locally at a particular rate, it must converge at least at that rate globally.}
\end{remark}
\subsection{Verification of the envelope convergence method for the whiplash dynamics for quadratic costs}\label{ver}
We now apply the envelope convergence method to the whiplash inertial gradient system for a particular class of convex functions. We define this set of functions as $f(x) = \frac{\lambda}{2} x^2$ for every $\lambda\in\mathbb{R}^{+}_{*}$, where $x\in \mathbb{R}$, to verify our envelope method, for which the system (\ref{wp}) reduces to the scalar ODE: 
    \begin{equation}
        \ddot{x}+\left(1+t \dot{x}^2\right)\dot{x}+ \lambda x = 0\quad\text{where}\;x\in\mathbb{R}.
    \label{quad}
    \end{equation}
Using our envelope convergence methodology, we choose a polynomial converger and start with a small value of $\eta$. To demonstrate the envelope convergence method for the system (\ref{quad}), we choose $\eta = 2$, using the polynomial convergence criterion as shown in the Figure: \ref{fig:pm}, which holds for every $\lambda\in\mathbb{R}^{+}_{*}$. 
\begin{figure}
    \centering
    \includegraphics[width = 0.65\columnwidth]{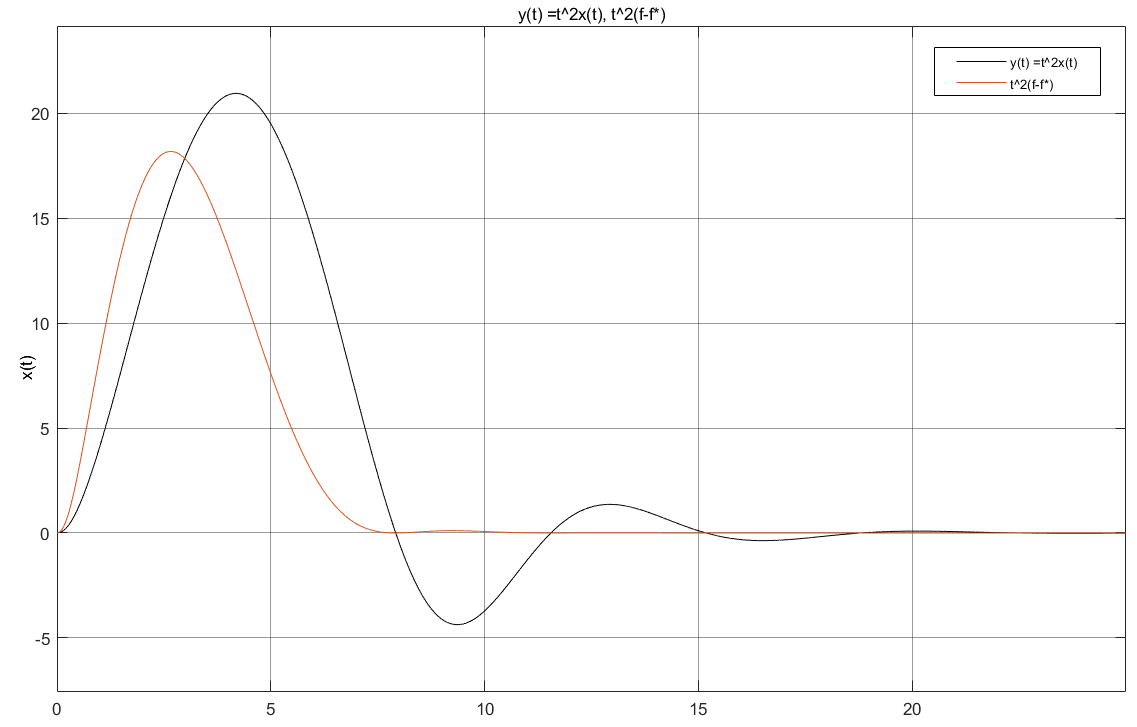}
    \caption{Verifying a polynomial convergence rate of $\mathcal{O}\left(\frac{1}{t^2}\right)$ with the cost function $f(x) = \frac{\lambda}{2} x^2$. Experimentally verified for variations on $\lambda$.}
    \label{fig:pm}
\end{figure}
For this case, we observe the transition to instability for the cost $\tilde{f}$ at $\eta = 0.495$ using the exponential rate as shown in Figure: \ref{fig:exp}.
\begin{figure}
    \centering
    \includegraphics[width = 0.65\columnwidth]{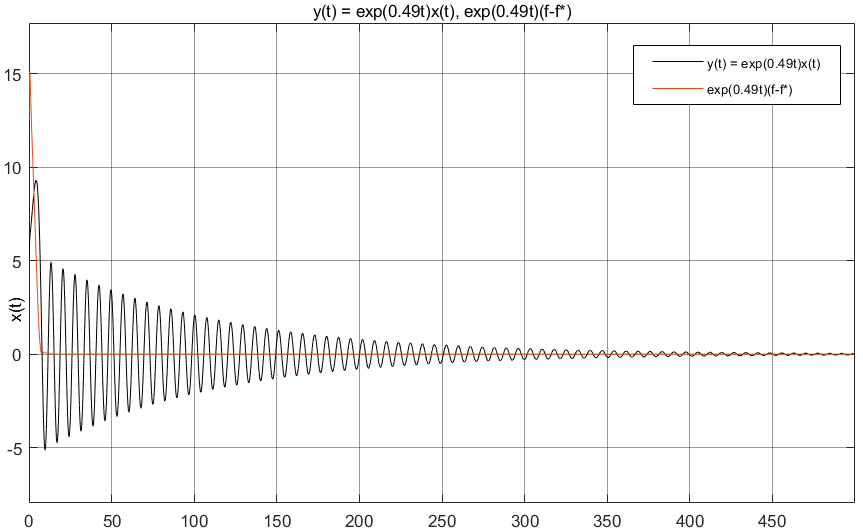}
    \caption{A numerical convergence rate of $\mathcal{O}\left(e^{-0.495t}\right)$ is verified for the whiplash gradient descent for the cost function, $\tilde{f}(x) = \frac{1}{2}x^2$.}
    \label{fig:exp}
\end{figure}
Since, we have found using computational methods\footnote{All numerical experiments have been performed in the SIMULINK\textsuperscript{\texttrademark} environment, using the \href{https://www.mathworks.com/help/simulink/gui/solver.html}{\color{blue}\texttt{Euler ode1 solver}} with a fixed step-size of 0.001.}, a suitable $\mathcal{P}(t) x(t)$, which converges, it is sufficient, therefore, to confirm the convergence rate of $\mathcal{O}\left(\frac{1}{t^2}\right)$ for the cost function $f = \lambda x^2$, or a rate of $\mathcal{O}\left(e^{-0.495t}\right)$ for the cost function $\tilde{f} = \frac{1}{2} x^2$, by showing that for the dynamics (\ref{quad}), there exists an asymptotic loosely bound solution such that it satisfies the design conditions (\ref{h1}) and (\ref{h2}) respectively. For our case, we thus have:
\begin{equation}
\begin{gathered}
    f(x)-f^* = \frac{\lambda}{2}x^2: x(t) = \mathscr{O}\left(\frac{1}{t^{\theta}}\right)\;\forall\;\theta>1.5,\\
    \tilde{f}(x)-\tilde{f}^* = \frac{1}{2}x^2: \tilde{x}(t) = \mathscr{O}\left(e^{-\tilde{\theta} t}\right)\;\forall\;\tilde{\theta}>0.2475.
\end{gathered}
\end{equation}
This means, using theorems \ref{tt} and \ref{leap}, it is sufficient to show that $\lim_{t\rightarrow\infty} t^{\theta}x(t) =0$ and $\lim_{t\rightarrow\infty} e^{\tilde{\theta}t}\tilde{x}(t) = 0$. Now, since $\mathcal{P}(t) x(t)$ is experimentally convergent, in both cases, it follows from (\ref{omega}), that both design criterion hold true for $\theta$ and $\tilde{\theta}$. This implies that $I_t$ is asymptotically loose bounded in both cases, and thus the required rates of convergence are established. Note that as per the requirement of theorem \ref{leap}, in (\ref{quad}), for the cost function $\hat{f}$, the control law $\alpha (t) = \mathscr{O}\left(\frac{e^{-pt}}{t}\right)$, where $p>0$. Though this does not agree explicitly with Assumption 5, the $\mathscr{O}$ loose upper bound enables that there exists some finite $r>0$, such that $\mathscr{O}\left(\frac{e^{-pt}}{t}\right)\asympl \mathscr{O}\left(e^{-rt}\right)$. Equivalent results can be observed for the multidimensional cost functions of $d \geq 2$. This has been shown in figures: \ref{fig:ni} and \ref{fig:nii}.

It is observed that there is no variance for convergence rate in finite time in variation of condition numbers\footnote{Note that we have to make the step-size smaller up to $s = 0.0001$ for this experiment to adjust for the stiffness of the problem.} for a time scaled error for the multi-dimensional function $f(x, y) = \frac{1}{2}\left(x^2+\kappa y^2\right)$. We have shown the plot of $t^2\|x-x^*\|$ with time $t$ in figure: \ref{fig:scale} for the envelope of $\mathcal{P}(t) = t^2$, all other conditions remaining the same. This observation suggests that the envelope function is consistent for our given setup, for a class of convex functions, numerically. Another distinct observation is that for the same step-size, the average error is higher for the well-conditioned function.
\begin{figure}
    \centering
    \includegraphics[width = 0.75\columnwidth]{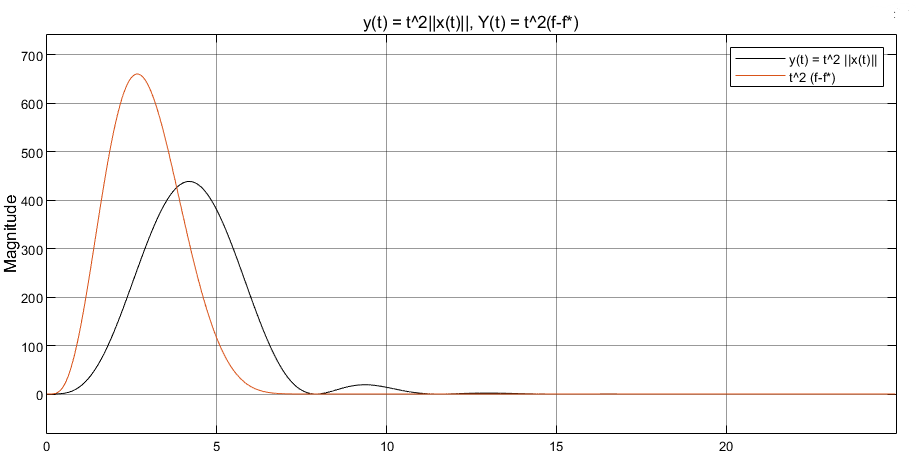}
    \caption{A polynomial convergence rate of $\mathcal{O}\left(\frac{1}{t^2}\right)$ with the cost function $F = \frac{\lambda}{2} (x^2+y^2)$ is observed. Experimentally verified for variations on $\lambda$.}
    \label{fig:ni}
\end{figure}
\begin{figure}
    \centering
    \includegraphics[width = 0.75\columnwidth]{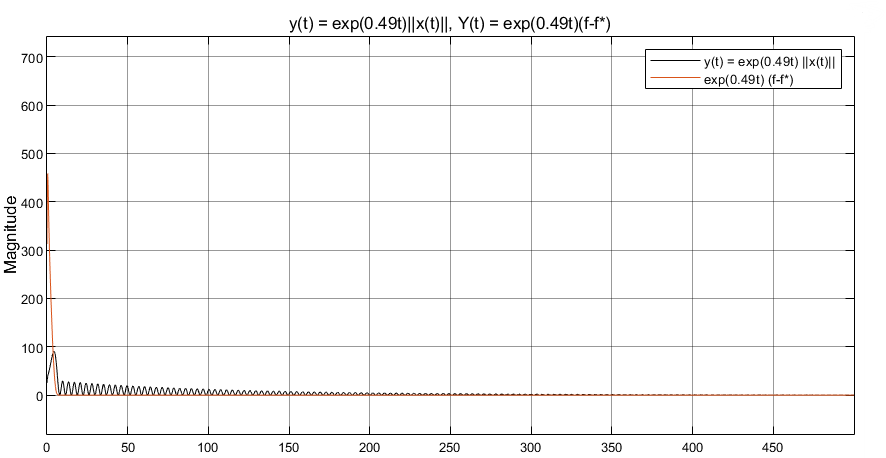}
    \caption{A numerical convergence rate of $\mathcal{O}\left(e^{-0.495t}\right)$ is verified for the whiplash gradient descent for the cost function, $\tilde{F} = \frac{1}{2}(x^2+y^2)$.}
    \label{fig:nii}
\end{figure}
\begin{figure}
    \centering
    \includegraphics[width = 0.75\columnwidth]{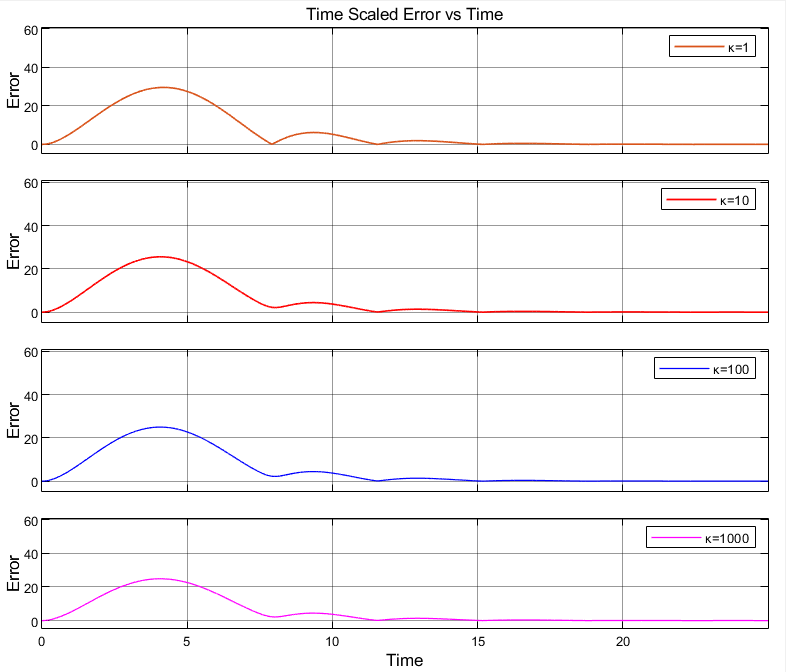}
    \caption{No variation in convergence time for the polynomial convergence rate of $\mathcal{O}\left(\frac{1}{t^2}\right)$ with the cost function $f(x, y) = \frac{1}{2} \left(x^2+\kappa y^2\right)$ upon varying the condition number $\kappa$. Error tracked with respect to time to an accuracy of $0.02\%$ of the minima.}
    \label{fig:scale}
\end{figure}

In summary, this method establishes using an integral anchored Lyapunov function, models for rates of convergences by using loose bound asymptotic analyses. The cornerstone of this approach lies in the numerical computation of a converger, which separates it from classical methods of establishing stability. In other words, we have reduced the analytical task of proving convergence rates by finding Lyapunov functions into the prediction of rates of convergence by finding appropriate convergers, which satisfy the critical conditions (\ref{h1}) or (\ref{h2}). This means that if $\nu\left(\theta,t\right)\left(x(t)-x^*\right)$ is convergent, then $\mathcal{P}\left(t, \eta\right)\left(f(x)-f^*\right)$ is convergent as well. Therefore, our analysis holds true as long as we choose a target $\eta$ for a suitable $\theta$, which does not violate the design criterion. This simplification provides a structured methodology for using experimental methods to check for the system's stability in a feasible time, without the necessity of resorting to finding Lyapunov functions --- specific to convergence rates, as elucidated in figure: \ref{fig:rate}.
\begin{figure}
    \centering    \includegraphics[width=0.6\columnwidth]{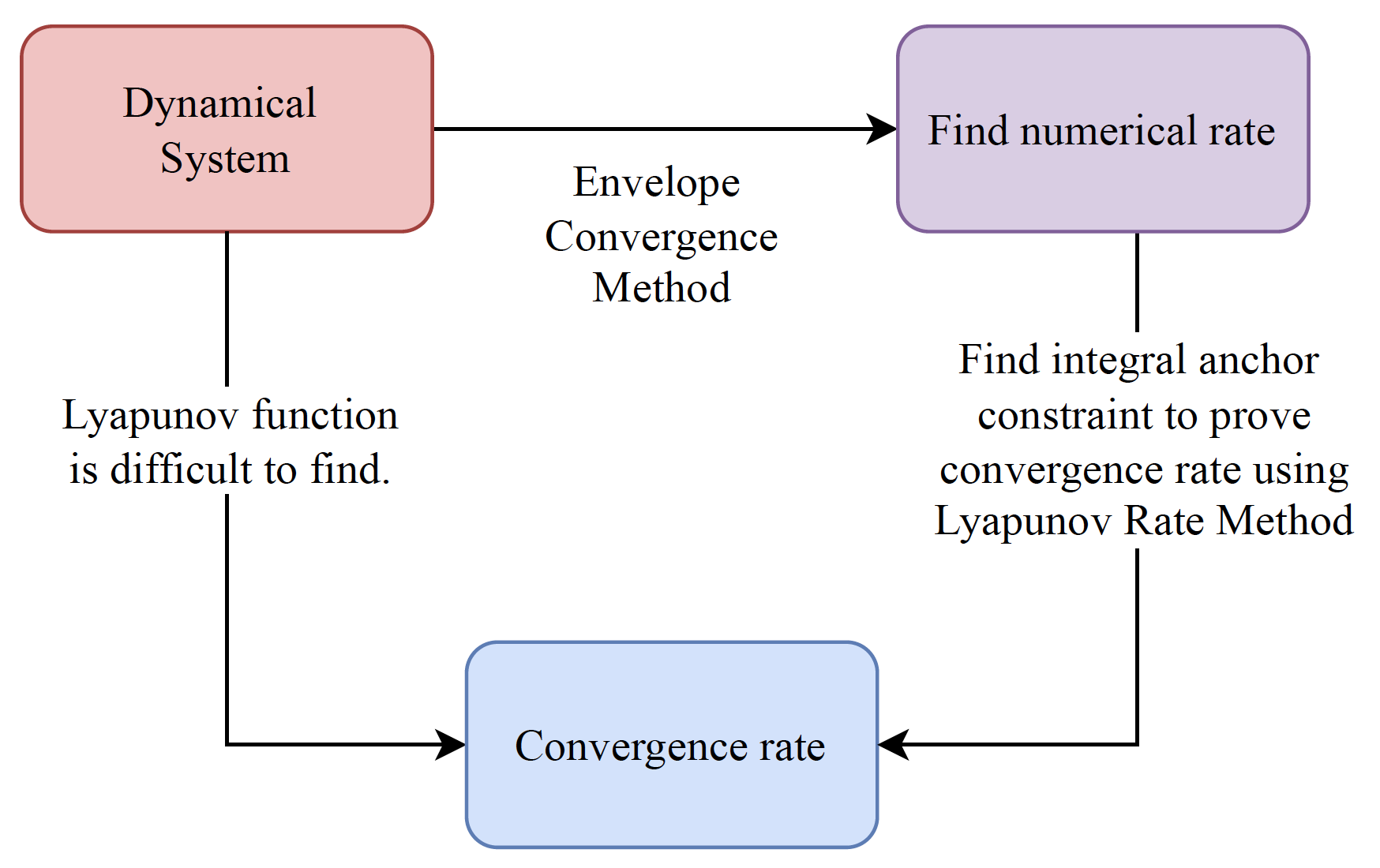}
    \caption{Method of Envelope convergence and Lyapunov Rate Method}
    \label{fig:rate}
\end{figure}
Note that every polynomial convergence rate can be expressed in terms of a weaker exponential rate but not vice-versa. In our analysis, it is inherent that a system globally converges but it is only with the use of the envelope method that we can test the global rates of convergence by studying the local convergence analysis of the time-scaled systems.

\section{Future Work}
\label{sec6}
As per the necessary  and sufficient condition for global optimization \cite{global}, to perform global optimization, one must inherently search globally (everywhere). The challenge of developing control frameworks in such cases is not merely computational but requires further investigation into the areas of global smoothness, evolution of attractors for time-varying systems and methods in infinite dimensional optimization. What we aim to achieve is the optimization of finite dimensional smooth functions which have guaranteed local convexity. This problem is two-pronged. 

First, finding convergence proofs of such methods become progressively complex with even more complex proofs for the rates of convergence, especially finding global minimas within a feasible frame of time. Thus, from an applied perspective in optimization, we need methods with guarantee numerically efficacy. This means all classical methods with limited guarantees of the functional settings, in which they are optimal are rendered sub-optimal in the face of non-convex settings, which is the industrial need. Secondly, it makes little sense to use deterministic algorithms in non-convex settings, which accelerate the task of finding one minima. The issue of saddle points in the context of deep learning, thus, becomes relevant, even more than local minimas. In \cite{exp}, the authors show that the probability of getting stuck in local minimas is exponentially rare. Instead algorithms which are computationally cheap and can steer out of local saddle points \cite{saddle} need to be further investigated. 

Methods built using the control framework, which are driven by systems' states or the local geometry of the cost function, inherently give rise to implicit schemes. This renders classical linear systems theory tools inappropriate for analysis because such schemes are inherently non-linear. Thus, in the future we need to develop tools to verify the efficacy of such methods which can adapt the structure of the feedback while being consistent to the energy limitations for global asymptotic stability in non-linear systems framework. This includes stochastic methods which can tune the noise with respect to the geometry or using adaptive control methods to address the structural uncertainty of the noisy oracles. 

Further theoretical work involves finding optimality criterion for implicit bounds on convergent sequences within bounded intervals in terms of passivity with sector-bound non-linearity \cite{Khalil}. Finally, an important component of our future work involves finding appropriate methods to optimize functions (appearing in deep learning) which have vanishing or exploding gradients \cite{vanish} or are non-smooth, using subgradient methods.

\section*{Acknowledgement}
The authors are grateful to the anonymous reviewers of the Asian Journal of Control for the detailed review report which allowed the authors to restructure the paper and correct a result. The authors are grateful to the School of Engineering, ANU for their support. Subhransu is grateful to Professor Emeritus Hédy Attouch, University of Montpellier, France, for productive discussions.
\bibliographystyle{IEEEtran}
\bibliography{Bib.bib}
\end{document}